%% file: 0-main.tex
\documentclass[twocolumn]{autart}

\usepackage{url}
\usepackage{graphicx}
\usepackage{mathtools}
\usepackage{amsmath, stmaryrd}
\usepackage{amssymb}
\usepackage{enumitem}
\usepackage{tabularray}
\setitemize{nolistsep}
\usepackage{caption}
\usepackage{subcaption}
\usepackage{pifont}
\usepackage{algorithm} 
\usepackage{algorithmic}  
\usepackage[ruled,vlined,algo2e,norelsize]{algorithm2e} 
\setenumerate{nolistsep}
\usepackage{todonotes}
\usepackage{utfsym}
\usepackage{booktabs}

\usepackage[framemethod=tikz,xcolor=true]{mdframed}
\usepackage{accents}

\DeclareMathOperator{\intr}{int}

\DeclareMathOperator{\linspan}{span}

\DeclareMathOperator{\PCont}{PC}
\DeclareMathOperator{\Cont}{C}

\DeclareMathOperator{\sbjto}{s. t.}
\DeclareMathOperator*{\argmin}{arg\,inf}

\DeclareMathOperator{\Ortho}{O}

\newcommand{\ra}{\ensuremath{\rightarrow}}

\newcommand{\lra}{\ensuremath{\longrightarrow}}
\newcommand{\eps}{\ensuremath{\varepsilon}}
\renewcommand{\le}{\ensuremath{\leqslant}}
\renewcommand{\ge}{\ensuremath{\geqslant}}
\renewcommand{\leq}{\le}

\newcommand{\cmark}{\ding{51}}  
\newcommand{\xmark}{\ding{55}}  

\newcommand{\dual}{^\star}

\let\emptyset\varnothing

\newcommand{\qedwhite}{\hfill \ensuremath{\Box}}

\newcommand{\abs}[1]{\ensuremath{\left\lvert{#1}\right\rvert}}

\newcommand{\inprod}[2]{\ensuremath{\left\langle{#1},\,{#2}\right\rangle}}

\newcommand{\secref}[1]{\S\ref{#1}}

\newcommand{\JOrtho}[1]{\Ortho_J}

\newcommand{\tangent}{T}
\newcommand{\dtangent}{\tangent\tangent}

\newcommand{\gradol}{\textsf{gradOL}}
\newcommand{\dlcvx}{\textsf{DLCvx}}

\newcommand{\odif}[1]{\mathrm{d} #1}

\newcommand{\xz}{\ensuremath{\overline{x}}}

\newcommand{\param}{\overline{x}}

\newcommand{\ol}{\overline}

\newcommand{\wt}{\widetilde}

\newcommand{\Let}{\coloneqq}
\newcommand{\teL}{\eqqcolon}
\newcommand{\nn}{\nonumber}

\newcommand{\Nz}{\ensuremath{\mathbb{N}}}
\newcommand{\N}{\ensuremath{\Nz^*}}

\newcommand{\norm}[1]{\ensuremath{\left\lVert #1 \right\rVert}}

\newcommand{\epower}[1]{\ensuremath{\mathrm{e}^{#1}}}

\newcommand{\manifold}{\mathsf{M}}

\newcommand{\pwcfunc}{\ensuremath{\mathsf{PWC}}}



\makeatletter
\renewcommand*\env@matrix[1][*\c@MaxMatrixCols c]{%
  \hskip -\arraycolsep
  \let\@ifnextchar\new@ifnextchar
  \array{#1}}
\makeatother

\DeclarePairedDelimiterX\lcrc[2]\lbrack\rbrack{#1,#2}
\DeclarePairedDelimiterX\lcro[2]\lbrack\lbrack{#1,#2}
\DeclarePairedDelimiterX\lorc[2]\rbrack\rbrack{#1,#2}
\DeclarePairedDelimiterX\loro[2]\rbrack\lbrack{#1,#2}

\renewcommand{\theequation}{(\arabic{equation})}
\makeatletter
\def\tagform@#1{\maketag@@@{\ignorespaces#1\unskip\@@italiccorr}}
\makeatother

\let\oldsqrt\sqrt
\def\sqrt{\mathpalette\DHLhksqrt}
\def\DHLhksqrt#1#2{%
\setbox0=\hbox{$#1\oldsqrt{#2\,}$}\dimen0=\ht0
\advance\dimen0-0.2\ht0
\setbox2=\hbox{\vrule height\ht0 depth -\dimen0}%
{\box0\lower0.4pt\box2}}

\newcommand{\Rbb}{\ensuremath{\mathbb{R}}}

\newcommand{\as}{\ensuremath{^\ast}}

\newcommand{\Dd}{\ensuremath{\mathcal{D}}}

\newcommand{\lpL}[1][]{\ensuremath{{L}^{#1}}}

\newcommand{\stpde}{\ensuremath{y}}

\newcommand{\spacedomainpde}{\ensuremath{\Omega}}

\newcommand{\uncertsetpde}{\ensuremath{\ol{\mathbb{W}}}}

\newcommand{\uncertpde}{\ensuremath{w}}

\definecolor{pinegreen}{rgb}{0.0, 0.545, 0.447} 

\newcommand{\distfield}{\ensuremath{v}}

\newcommand{\st}{\ensuremath{x}}
\newcommand{\cont}{\ensuremath{u}}

\newcommand{\tinit}{0}
\newcommand{\tfin}{T}

\newcommand{\rcost}{\ensuremath{c}}
\newcommand{\fcost}{\ensuremath{c_F}}

\newcommand{\horizon}{\ensuremath{T}}

\newcommand{\admcont}{\ensuremath{\mathbb{U}}}
\newcommand{\J}{\ensuremath{\mathbb{J}}}

\newcommand{\dimst}{\ensuremath{d_x}}
\newcommand{\dimcon}{\ensuremath{d_u}}

\DeclarePairedDelimiterX\aset[1]\lbrace\rbrace{\def\suchthat{\; \delimsize\vert\;}#1}

\newcommand{\admst}{\mathbb{X}}
\newcommand{\admcon}{\mathcal{U}}
\newcommand{\valf}{\ensuremath{F_{\circ}^{*}}}
\newcommand{\fsblset}{\ensuremath{\mathbb{X}_T}}
\newcommand{\quito}{\textsf{QuITO}}

\newcommand{\rzero}{\ensuremath{\rho}}

\DeclareFontFamily{U}{mathx}{}
\DeclareFontShape{U}{mathx}{m}{n}{<-> mathx10}{}
\DeclareSymbolFont{mathx}{U}{mathx}{m}{n}
\DeclareMathAccent{\widecheck}{0}{mathx}{"71}

\newcommand{\slack}{\ensuremath{s}}

\newcommand{\tseq}{\mathcal{T}}

\newcommand{\dvar}{\ensuremath{\widetilde{N}}}
\newcommand{\feasSip}{\mathcal{S}}
\newcommand{\dict}{\ensuremath{\mathcal{D}}}

\newcommand{\Param}{\ensuremath{\alpha}}
\newcommand{\Paramset}{\ensuremath{\mathcal{S}_{u,s}}}
\newcommand{\Paramw}{\ensuremath{\beta}}
\newcommand{\Paramv}{\ensuremath{\gamma}}
\newcommand{\Reg}{\ensuremath{\Psi}}
\newcommand{\reg}{\ensuremath{\psi}}

\newcommand{\gfunc}{\ensuremath{\mathcal{G}}}

\newcommand{\valuefunc}{\ensuremath{F_{\circ}^{*}}}

\newcommand{\glbopt}{\textsf{GlbOpt}}

\newcommand{\regul}{\textsf{Regu}}

\newcommand{\objective}{\ensuremath{F_{\circ}}}
\newcommand{\objectivedisc}{\ensuremath{\widehat{F}_{\circ}}}

\begin{document}

\begin{frontmatter}


\title{Exact solutions to a class of constrained optimal control problems via lossless convexification for digital control \thanksref{footnoteinfo}} 

\thanks[footnoteinfo]{The first two authors contributed equally. This paper was not presented at any IFAC 
meeting.}

\author[1]{Vaibhav Upadhyay}\ead{vaibhav.u@iitb.ac.in},    
\author[2]{Siddhartha Ganguly}\ead{sganguly41@gatech.edu},               
\author[1]{Debasish Chatterjee}\ead{dchatter@iitb.ac.in}  

\address[1]{Centre for Systems and Control, Indian Institute of Technology Bombay, India}  
\address[2]{Daniel Guggenheim School of Aerospace Engineering, Georgia Institute of Technology, Atlanta, USA}             
          

\begin{abstract}                          
This article establishes a new numerically viable technique for solving a class of constrained, nonconvex, continuous-time optimal control problems (OCPs) for linear systems that commonly arise in aerial and aerospace applications. The lossless convexification technique is employed to translate the original nonconvex OCP with annular control magnitude constraints into a convex problem, and then by finitely parametrizing the control space with piecewise constant functions, an efficient numerical approach is established that guarantees \emph{exact} solutions while ensuring the satisfaction of an uncountable family of constraints over a compact time interval. The effectiveness of the approach is demonstrated on a spacecraft landing problem involving three degrees of freedom (DoF), underscoring its potential for real-world aerospace guidance and control tasks.
\end{abstract}

\end{frontmatter}

\input{Introduction}
\input{Problem_Setup}
\input{main_results}
\input{Numerics}


\bibliographystyle{plain}       
\bibliography{refs}           

\end{document}

%% file: Introduction.tex
\section{Introduction}
\label{s:intro}
This article addresses a challenging issue in finite-horizon optimal control problems (OCPs) for linear systems under a class of annular (nonconvex) control constraints. The particular type of constraints addressed in this work arises in a wide range of applications, including energy- and fuel-optimal trajectory planning of aerial and aerospace vehicles \cite{ref:ZSGZHHCHYWFYW-23}, automotive systems \cite{ref:WasKolSteDel-14}, and the optimal control of multiplexed and overactuated systems \cite{ref:MWH-21}. Key challenges that arise in the context of the numerical viability of OCPs with upper and lower bounds on the control magnitude are primarily twofold: (a) due to the ensuing nonconvexity, standard direct numerical methods that discretize the original infinite-dimensional OCP into a finite-dimensional nonlinear programming (NLP) problem retains its nonconvexity and frequently encounter feasibility and convergence issues \cite{ref:AWLTB-06,ref:Yakub-97,ref:SedOdenDir-23}, and (b) accounting for a compact but uncountable family of inequality constraints in a computationally viable fashion.

The so-called \emph{lossless convexification approach} offers an alternative viewpoint: It identifies a convex problem (typically a relaxation) whose optimal solution is identical to the optimal solution to the original problem. A growing body of work is resorting to this approach. For instance, nonconvex annular control constraints were studied in \cite{ref:MWHBA-14b}; \cite{ref:YRLXLD-24} examined nonconvex equality alongside linear ones; an extreme-point convex relaxation approach for lossless convexification was introduced in \cite{ref:MWH-21}; and for linear systems \cite{ref:YRLX-23} demonstrated that a convex-hull-based relaxation provides an alternative convexification strategy while retaining the original solution. Related approaches based on successive convexification for nonconvex optimal control, as well as discrete-time formulations of lossless convexification, were considered in \cite{ref:SuccConv-25,ref:Dlcvx-26}. These prior contributions conform to the category of simplifying the nonconvex problem to an equivalent convex one and then solving the latter computationally. The issue (b), pointed out in the preceding paragraph, has remained unaddressed yet.


\subsection*{Contributions}
Against the preceding backdrop, our efforts are directed at developing a computationally viable flowchart for \emph{exact solutions} to OCPs for linear systems with nonconvex annular constraints and their digital implementations. This class of OCPs is of special importance in aerospace applications. More precisely:
\begin{enumerate}[label=\textup{(C-\alph*)}, leftmargin=*, widest=b, align=left]
    \item \label{contrib:1} We employ the method of \emph{lossless convexification} \cite{ref:MWHBA-14b} to translate the given nonconvex OCP into an equivalent convex OCP. The resulting convexified optimal control problem becomes an infinite-dimensional constrained convex program, generally intractable by numerically viable means, with exact solutions being well beyond the reach of computational methods. Towards this end, for computational viability we employ a piecewise constant parametrization of the admissible controls fine-tuned to handshake with digital control implementations. The (uniform) time-step of the discretization is the user's choice and may be as small as permissible in alignment with the hardware resources at hand.

	\item \label{contrib:2}  Although the ensuing optimal control problem is finitely parametrized, it still stipulates the satisfaction of an uncountable family of constraints. To address this, we establish a novel approach that employs numerical solutions to convex semi-infinite programs (CSIPs) to solve such problems in an \emph{exact fashion}, guaranteeing constraint satisfaction over the entire uncountable family. Our approach strategically selects a finite set of time instants and enforces the constraints at these points, but this selective enforcement is shown to be sufficient to guarantee satisfaction of the \emph{entire uncountable set} of constraints across the complete time horizon. Moreover, we establish that the optimal value and associated optimizers of the CSIP can be recovered exactly by solving a finite and relaxed convex optimization problem.
\end{enumerate}

\subsection*{Perspective and questions}
Central to our contribution is the following question: \emph{Are existing numerical methods sufficient for this class of problems?} Not always. Indeed, standard numerical techniques for optimal control --- such as direct multiple shooting and direct collocation \cite{ref:betts-book,ref:SGNRDCRB-22,ref:QuITOv2} --- adhere to the classical discretize-then-optimize paradigm. These techniques enforce the satisfaction of the constraints at finitely many time-points over the horizon (typically the points of discretization) that \emph{may not} imply constraint satisfaction for \emph{all} time instants; this can turn out to be detrimental in safety-critical applications. Indeed, in \secref{subsec:infeas_fuel_opt} we demonstrate the infeasibility of a particular problem that remained undetected via techniques reported in the literature, whereas our technique reports its infeasibility quite readily. By enabling the enforcement of a compact uncountable family of convex constraints, our framework provides guarantees of exact solutions, offering, to our knowledge, the first such result in numerical optimal control, and demonstrates a level of numerical fidelity that is generally not achieved by conventional direct methods. 

\emph{What is special about the piecewise constant parametrization?}
The adoption of a piecewise constant parametrization for control trajectories serves, first, to transform an inherently infinite-dimensional optimization problem into a finite-dimensional mathematical program; this is inevitable, especially when addressing general constrained problems \cite[Chapter 16, Section 1]{ref:AliTsa-21}. Second, this particular parametrization is especially amenable to digital implementation, since its inherently discrete and quantized structure maps naturally to fixed-precision hardware, facilitating efficient, reliable, and faithful execution.

\textbf{Purpose, limitations, and applicability of our algorithm:} As mentioned in \eqref{contrib:1}--\eqref{contrib:2}, the primary contribution and overarching goal of our architecture is to develop a computationally viable, robust-optimization–driven pipeline for a class of continuous-time optimal control problems, providing guarantees of (i) exact solutions and (ii) continuous-time satisfaction (certification) of safety constraints.

That said, these guarantees come at the cost of computational time. In particular, the architecture includes a global-optimization module, which is the dominant source of complexity; see \S\ref{subsec:algo} and Footnote 7 as well. Consequently, from an applications standpoint, the current framework is \emph{not} well-suited for online implementations such as model predictive control (MPC). Rather, it is primarily intended for \textbf{offline trajectory generation}, where \textbf{solution accuracy, feasibility, and continuous-time constraint certification} are critical, and substantial computational resources can be allocated. Representative settings include powered-landing and re-entry guidance design, and safe trajectory synthesis for quadrotor landing, where brief inter-sample violations of safety or actuator constraints are unacceptable, and offline computation is routinely used to obtain certified feasible trajectories.

\subsection*{Notation}
For us \(\N \Let \aset{1,2,\ldots}\) is the set of positive integers and \(\Nz\Let \N \cup \aset[]{0}\). If \(m_1, m_2 \in \N\) and \(m_1<m_2\), then \([m_1;m_2]\Let \aset[]{m_1,m_1+1,\ldots,m_2}\). The vector space \(\Rbb^{d}\) is equipped with standard inner product, \(\inprod{x}{y}\Let \sum_{j=1}^d x_j y_j\) for every \(x,y \in \Rbb^{d}\). For a fixed \(\horizon > 0\) and \(n \in \N\), the space of \(\Rbb\)-valued piecewise constant functions with \(n\) uniform segments over \(\lcrc{0}{\horizon}\) is defined as  
\( \pwcfunc_n(\lcrc{0}{\horizon};\Rbb) = \{ u: \lcrc{0}{\horizon} \to \Rbb \mid u(t) = u_k, \text{ for } t \in \lcrc{t_k}{t_{k+1}}, \, k = 0, \dots, n-1 \} \),  
with uniform grid points \(t_k \Let \frac{k\horizon}{n}\), and \(k = 0, \ldots, n \).  

%% file: Problem_Setup.tex
\section{Problem formulation}\label{sec::prob_form}
Let \(\dimst,\dimcon, d_r, m \in \N\) and fix a time horizon \(\tfin>0\). Let us consider a linear time-invariant controlled dynamical system
\begin{equation}
	\label{eq:sys}
	\dot \st(t) = A\st(t) + B\cont(t) + \distfield\,\,\text{for all\ } t \in \lcrc{0}{T}, 
\end{equation}
where \(\lcrc{0}{\horizon} \ni t \mapsto \st(t) \in \Rbb^{\dimst}\) is the absolutely continuous state trajectory and \(\lcrc{0}{\horizon} \ni t \mapsto \cont(t) \in \Rbb^{\dimcon}\) is the locally integrable control trajectory \cite[p. 6]{ref:Zab-20}; \(\distfield \in \Rbb^{\dimst}\) represents the external constant vector field.\footnote{The notion of absolute continuity is standard in analysis; see \cite[Chapter 3, \S 3.5, p. 105]{ref:FolReal-99} and \cite{ref:Zab-20} for a control theoretic context. Adapted to our setting: For every \(\eps'>0\) there exists \(\delta>0\) such that for any finite collection of pairwise disjoint open intervals
\(\aset[\big]{\loro{a_k}{b_k}}_{k=1}^{{L}} \subset \lcrc{0}{\horizon}\),
\begin{align}
\sum_{k=1}^{L} (b_k-a_k) < \delta \;\Longrightarrow\;
\sum_{k=1}^{L} \norm{\st(b_k) - \st(a_k)} < \eps',  \nn 
\end{align}
where \(\norm{\cdot}\) denotes the Euclidean norm on \(\Rbb^{\dimst}\).} The matrices \(A \in \Rbb^{\dimst \times \dimst}\) and \(B \in \Rbb^{\dimst \times \dimcon}\) represent the state and actuation matrices, respectively. The boundary and path constraints associated with \eqref{eq:sys} are
\begin{align}
    \st(\tinit) & \Let \param \in \Rbb^{\dimst},\,\, r_{F}(\st(\horizon)) = 0, \label{eq:state_constraint_a} \\
    \st(t) &\in \admst \,\, \text{ for all } t \in \lcrc{0}{\horizon}, \label{eq:state_constraint_b}
\end{align}
and the corresponding OCP is
\begin{equation}
	\label{eq:OCP}
\begin{aligned}
& \inf_{\cont(\cdot) \in \,\admcon }	&&  \objective(\param,\cont(\cdot))\Let \fcost\bigl(\st(\tfin)\bigr) + \int_{\tinit}^{\horizon} \rcost\bigl(\kappa(\cont(t))\bigr) \odif{t} . \\
&  \sbjto		&&  \hspace{-2mm}\begin{cases}
\text{dynamics }\eqref{eq:sys}, \text{ constraints \eqref{eq:state_constraint_a} and \eqref{eq:state_constraint_b}},\\
\text{and } u(t) \in \admcont\text{ for all } t \in \lcrc{0}{\horizon},\\
\end{cases}
\end{aligned}
\end{equation}
with the following data:
\begin{enumerate}[label=\textup{(4-\alph*)}, leftmargin=*, widest=b, align=left]

\item \label{OCPdata1} The map \(\kappa: \Rbb^{\dimcon} \lra \lcro{0}{+\infty}\) is near-monotone, Lipschitz continuous, convex, and strictly positive except possibly at \(0\).\footnote{A function \(f:\Rbb^{\nu}\lra \Rbb\) is near-monotone if \(\lim_{\|\xi\| \to +\infty}f(\xi)  = +\infty\).} The set \(\admcont\) is a nonempty compact subset of \(\Rbb^{\dimcon}\) such that
\(\admcont \Let \{u \in \Rbb^{\dimcon} \mid \rho_1 \leq \kappa(u) \leq \rho_2,\,\, Cu \leq d\}\)
with \(0 < \rho_1 < \rho_2 < +\infty \), where \( C \in \Rbb^{m \times \dimcon} \) and \( d \in \Rbb^m \). The control trajectory takes values in 
\begin{equation}
	\hspace{-2mm}\admcon\Let \aset[\big]{u \in \mathsf{PC}(\lcrc{0}{\horizon};\admcont)\suchthat \cont(t) \in \admcont\,\,\text{for all}\,\,t \in \lcrc{0}{T}}, \nn
\end{equation}
where \(\mathsf{PC}(\lcrc{0}{\horizon};\admcont)\) stands for \(\admcont\)-valued piecewise continuous function is defined on \(\lcrc{0}{\horizon}.\)

\item \label{OCPdata2} The \emph{instantaneous cost} \(\rcost:\lcro{0}{+\infty} \lra\lcro{0}{+\infty} \) is convex and strictly positive except possibly at \(0\). Moreover, \(\rcost(\cdot)\) is locally Lipschitz and the \emph{terminal cost} \(\fcost(\cdot)\) is affine.

\item \label{OCPdata3} The map \(r_F: \Rbb^{\dimst} \lra \Rbb^{d_r}\) is affine and the admissible state set \(\admst  \subset \Rbb^{\dimst}\) is restricted to be a fixed linear subspace of \(\Rbb^{\dimst}\). 
\end{enumerate}
\vspace{0.5mm}
A control trajectory \(\cont(\cdot)\) is \textit{feasible} if it belongs to the class \(\mathcal{U}\), satisfies the control constraint in \eqref{OCPdata1}, and the corresponding solution \(\st(\cdot)\) of \eqref{eq:sys} satisfies \(r_F(\st(\horizon)) = 0\). We assume that the interior of the feasible set associated with the OCP \eqref{eq:OCP} is nonempty, i.e., there exists at least one admissible process \(\bigl(x\as(\cdot),u\as(\cdot)\bigr)\) for which \(\objective(\param,\cont\as(\cdot)) < +\infty\); we call such a \(\cont\as:\lcrc{0}{\horizon} \lra \admcont\) is an \emph{optimal control trajectory}.

\subsection{Convex relaxation and optimal solution}\label{subsec:lossconv}
Observe that the control constraint set \(\admcont\) in \eqref{eq:OCP} is nonconvex; indeed in \(\rho_1 \le\kappa(\cont(t)) \le\rho_2,\) the lower bound \(\kappa(u(t)) \ge \rho_1\) makes the set of admissible control actions annular and hence nonconvex. This subsection introduces our relaxation procedure, starting with a new decision function \(\slack: \lcrc{0}{\horizon} \lra [0, +\infty[ \), which we will refer to as the slack variable; we reformulate the nonconvex control constraint in \eqref{eq:OCP} into
\begin{align}
    \label{eq:conv_constraint}
    \begin{cases}
        \slack(t) \in [\rho_1, \rho_2],\,\, C\cont(t) \leq d, \text{ and} \\ 
        \kappa(u(t)) \leq \slack(t) \text{ for all } t \in \lcrc{0}{\horizon}\\
    \end{cases}
\end{align}
and the reformulated OCP is 
\begin{equation}
	\label{eq:COCP}
\begin{aligned}
&\hspace{-2mm} \inf_{\cont(\cdot), \slack(\cdot)}	&& \hspace{-3mm} \objectivedisc(\param,\cont(\cdot), \slack(\cdot)) \Let \fcost\bigl(\st(\tfin)\bigr) + \int_{\tinit}^{\horizon} \rcost\bigl(\slack(t)\bigr) \odif{t} \\
&  \sbjto		&&  \hspace{-4mm}
\text{dynamics }\eqref{eq:sys}, \text{ constraints \eqref{eq:state_constraint_a}, \eqref{eq:state_constraint_b} and \eqref{eq:conv_constraint}.}\\
\end{aligned}
\end{equation}
The OCP \eqref{eq:COCP} is convex because, after introducing the slack variable \(\slack(\cdot)\), all decision variables enter through convex constraints and a convex objective. In particular, the lifted constraint \(\kappa(\cont(t)) \le \slack(t)\) is convex since \(\kappa(\cdot)\) is convex and this is an epigraph-type inequality; the bounds \(\rho_1 \le \slack(t) \le \rho_2\) and the linear control constraints are affine. The dynamics are linear, so the state depends affinely on the control, and imposing convex state/path constraints pointwise in time preserves convexity. Finally, the objective in \eqref{eq:COCP} is convex in \(\slack(\cdot)\), so the overall optimization problem is convex.

\begin{assum}
    \label{COCP_criteria}
    There exist feedback matrices \(F\) and \(G\) (named \emph{friends} in \cite{ref:HLT-02}) such that the set \(\admst\) in \eqref{OCPdata3} is the strongly controllable subspace for the linear system \((A+BF,BG,CF,CG)\), where \(A\) and \(B\) are the system matrices in \eqref{eq:sys}, and \(C\) is the control constraint matrix defined in \ref{OCPdata1}.
\end{assum}
Assumption \ref{COCP_criteria} provides a sufficient condition \cite{ref:MWHBA-14b} ensuring that solutions to \eqref{eq:COCP} are also solutions to \eqref{eq:OCP}.
Note that the existence of optimal solutions to \eqref{eq:COCP} can be distilled directly from the problem data \eqref{OCPdata1}--\eqref{OCPdata3} by applying \cite[Theorem 23.11]{ref:FC-13}. The next theorem quotes a standard result of \emph{lossless convexification}.
\begin{thm}\cite[Theorem 2]{ref:MWHBA-14b}
If Assumption \ref{COCP_criteria} is satisfied in the context of \eqref{eq:OCP} along with its associated data \eqref{OCPdata1}--\eqref{OCPdata3}, then optimal solutions of the OCP \eqref{eq:COCP} are also the optimal solutions of the original OCP \eqref{eq:OCP}. 
\end{thm}
\begin{rem}
Regardless of the convexification approach used, whether that of \cite{ref:MWHBA-14b}, \cite{ref:YRLX-23}, or \cite{ref:MWH-21}, our method in \S\ref{sec:refSIP} ahead, establishes \emph{exact} solutions of convexified OCPs in continuous time and enforces the uncountably infinite family of constraints in \eqref{eq:COCP}.
\end{rem}

\subsection{Numerically viable reformulation of \eqref{eq:COCP}}
The OCP \eqref{eq:COCP} is inherently infinite-dimensional since \(\cont(\cdot)\) and \(s(\cdot)\) are infinite-dimensional objects. To account for finite arithmetic and digital computer implementations, we proceed via restricting the admissible space to finite-dimensional subspaces, where the control and the slack trajectories are parametrized by functions in \(\pwcfunc_N(\lcrc{0}{\horizon};\Rbb)\); this is a standard approach in numerical optimal control \cite{ref:DasGanAnjCha23CDC}. 

Let \(N \in \N\), and consider a dictionary \(\dict \Let \bigl(\psi_i(\cdot)\bigr)_{i \in \Nz} \subset  \pwcfunc_N(\lcrc{0}{\horizon};\Rbb)\) that serves the purpose of naturally parametrizing the admissible controls (amenable to digital implementation) and slack trajectories. We define the set of admissible generating functions
\begin{align*}
    \admcon_{\dict} \Let \linspan \aset[\big]{ \psi_i:[0,\horizon] \ra \Rbb \suchthat i = 1, \ldots, N },
\end{align*}  
where the piecewise constant and linearly independent generating functions \(\reg_i(\cdot)\) are identified and fixed as  
\begin{align}
\label{eq:param_basis}
    t \mapsto \reg_i(t) &\Let
    \begin{cases}
        1 & \text{if } t \in \left[ (i-1)\frac{T}{N}, i\frac{T}{N} \right[, \\
        0 & \text{otherwise},
    \end{cases}
\end{align}  
for all \(i \in [1;N]\). Note that \(\max_{i \in \Nz}\max_{t}|\psi_i(t)| = 1\) by design. The control trajectory \(\lcrc{0}{\horizon} \ni t \mapsto \cont(t) \in \Rbb^{\dimcon}\) is represented component-wise, where the \(i\)-th component of the control trajectory is a linear combination of generating functions \(t \mapsto \psi_j(t)\) for \(j \in [1;N]\). Define the function \(t \mapsto \Reg(t) \Let \bigl(\psi_1(t),\, \psi_2(t),\;\ldots \;\psi_N(t) \bigr) \in \Rbb^{N}\). Then for each \(i=1,\ldots,\dimcon\) we have the component-wise parametrization
\begin{align} 
    \label{e:pcontrol:1}
    [0,\horizon] \ni t \mapsto \cont^{\dict}_i(t)= \sum_{j=1}^{N}\alpha_{i,j} \psi_j(t) \teL \inprod{\Param_i}{\Reg(t)}, 
\end{align}
where \(\alpha_{i,j} \in \Rbb\) are the control parameters. Consequently, one can write the control trajectory \(t \mapsto \cont^{\dict}(t)\) in the compact form
\begin{align}
    \label{e:cont_param}
    \cont^{\dict}(t) = \begin{pmatrix}
        \Param_{1,1} & \Param_{1,2} & \cdots & \Param_{1,N}\\
        \Param_{2,1} & \Param_{2,2} & \cdots & \Param_{2,N}\\
        \vdots & \vdots & \ddots & \vdots\\
        \Param_{\dimcon,1} & \Param_{\dimcon,2} & \cdots & \Param_{\dimcon,N}
    \end{pmatrix} \hspace{-1mm}\Reg(t) \teL \Param \Reg(t),
\end{align}
with \(\Param \in \Rbb^{\dimcon \times N}\) a coefficient matrix (to be determined). Following the parametrization in \eqref{e:pcontrol:1}, we express the slack variable trajectory \( \lcrc{0}{\horizon} \ni t \mapsto \slack(t)\in \Rbb\) also as a linear combination of the generating functions (keeping the same notation) \(t \mapsto \reg_i(t) \) for \(i \in [1;N]\) 
\begin{align} 
    \label{e:slack_param}
    [0,\horizon] \ni t \mapsto \slack^{\dict}(t)= \sum_{i=1}^{N}\beta_{i} \reg_i(t) \teL \inprod{\beta}{\Reg(t)}, 
\end{align}
where \(\beta_i \in \Rbb\) are the slack coefficients (to be determined). In view of the parametrization \eqref{e:pcontrol:1} and \eqref{e:slack_param}, the control and slack constraints are given by 
\begin{align}
    \label{eq:cont_param_constraint}
    \begin{cases}
        \kappa(\cont^{\dict}(t)) \leq \slack^{\dict}(t),  \,\, C\cont^{\dict}(t) \leq d,\\ 
        \text{and } \rho_1 \leq \slack^{\dict}(t) \leq \rho_2 \,\, \text{for all}\,\, t\in [0,\horizon].
    \end{cases}
\end{align}
Note that the solution of \eqref{eq:sys}, starting from an initial state \(\param \in \Rbb^{\dimst}\), is given by the variation of constants formula \cite[\S 2.3]{ref:AB-20}, i.e., for all \(t \in \lcrc{0}{\horizon}\)
\begin{align}\label{eq:p_sol}
\st(t;\param,\Param) &=  e^{At}\param +  \int_{0}^{t} e^{A(t-\tau)} \left(B \Param \Reg(\tau) + \distfield\right)  \odif{\tau}.
\end{align}
With these components in place, the resulting finitely parametrized optimal control problem is given by
\begin{equation}
	\label{e:fOCP} 
\begin{aligned}
& \hspace{-2mm}\inf_{(\cont^{\dict}(\cdot), \slack^{\dict}(\cdot))}	&& \hspace{-2mm}\objectivedisc\bigl(\param,\cont^{\dict}(\cdot), \slack^{\dict}(\cdot)\bigr) \\
&  \hspace{0mm}\sbjto	&&  \hspace{-2mm}
\text{constraints \eqref{eq:state_constraint_a}, \eqref{eq:state_constraint_b} and \eqref{eq:cont_param_constraint}}.
\end{aligned}
\end{equation}
\begin{assum}
\label{assum:slater_cond}
The OCP \eqref{e:fOCP} is strictly feasible in the sense that Slater's condition holds.
\end{assum}

\begin{rem}\label{rem:slater}
Slater's condition is a technical requirement needed in Assumption \ref{assum:slater_cond} needed to ensure strict feasibility of the OCP \eqref{e:fOCP}, and it is employed in the proof of our main theorems ahead (Theorems \ref{thm:sampling_thm} and \ref{thrm:opt:extract}). Assumption \ref{assum:slater_cond} implies that: there exists \(\bigl(\cont^{\dict}(\cdot), \slack^{\dict}(\cdot)\bigr)\) with \(\cont^{\dict}(\cdot)\) and \(\slack^{\dict}(\cdot)\) as in \eqref{e:pcontrol:1} and \eqref{e:slack_param}, and a nonempty open set \(O \subset \intr{\admst}\), such that 
\begin{align*}
\begin{cases}
\st(t) \in O, \;\st(0)=\param, \;r_F(\st(\tfin))=0,\\
\kappa(\cont^{\dict}(t)) < \slack^{\dict}(t),\, C\cont^{\dict}(t) < d,\\ 
\text{and } \rho_1 < \slack^{\dict}(t) < \rho_2 \text{ for all } t\in \lcrc{0}{\horizon}.
\end{cases}
\end{align*}
\end{rem}
Since \(\param\) appears as a parameter in \eqref{e:fOCP}, we define \(\fsblset\) as the set of initial states \(\param\) for which \eqref{e:fOCP} remains strictly feasible. The optimal value mapping is then given by
\begin{equation}\label{eq:value funcion}
\fsblset \ni \param \mapsto  \valf(\param) \Let \text{the optimal value of \eqref{e:fOCP}}.
\end{equation}
\begin{rem}
We revisit an important issue highlighted in the introduction: Since the OCP \eqref{eq:COCP} is an infinite‑dimensional constrained optimization problem, computing an exact solution is, in general, intractable \cite[Chapter 16, \S I]{ref:AliTsa-21,ref:XiaYangWang-21}. All numerical techniques must employ some form of parametrization to reduce the infinite-dimensional decision space of \eqref{eq:COCP} to a finite-dimensional one. We adopted the component-wise parametrization of the control trajectory afforded by the piecewise constant functions \(\pwcfunc_N(\lcrc{0}{\horizon};\Rbb)\) along the lines of \cite{ref:MWHBA-14a,ref:QLRLKLT-14}, that is well-suited to digital implementation. Since any optimal control trajectory is locally integrable \cite[Chapter 23, \S 23.3]{ref:FC-13} and the family of piecewise‑constant functions is dense in \(L^1(\lcrc{0}{\horizon};\admcont)\) \cite[Chapter 6]{ref:FolReal-99}, for sufficiently large \(N\) it admits an arbitrarily close approximation (in the $L^1$ norm) by elements of \(\pwcfunc_N\bigl(\lcrc{0}{\horizon};\Rbb\bigr)^{\dimcon}\). A key point of departure of our method is the enforcement of the constraint at every instant of time as opposed to at \emph{only} the points of time discretization commonly adopted in the literature. Nevertheless, our results unveil a new technique capable of solving \eqref{e:fOCP} while staying in the finitary regime of memory and computational power. This is the subject of section \S\ref{sec:refSIP}. This feature enables our method to uncover infeasibilities that standard approaches enforcing constraints only at the discretization nodes are liable to miss; see \S\ref{subsec:infeas_fuel_opt} for an important illustration on this front.
\end{rem}

%% file: main_results.tex
\section{Convex SIP formulation and numerical solution}
\label{sec:refSIP}
The finitely parametrized OCP \eqref{e:fOCP} turns out to be a convex semi-infinite program (CSIP), and demonstrating this is the purpose of this section. Furthermore, leveraging the CSIP structure, one can numerically solve \eqref{e:fOCP} in an \emph{exact manner} (see our main results, Theorem \ref{thm:sampling_thm} and Theorem \ref{thrm:opt:extract} ahead for precise statements) using recently developed exact numerical algorithms for solutions to CSIPs. Observe that the constraints in \eqref{eq:cont_param_constraint} after parametrization are
\begin{align}
\label{constr:control}
\begin{cases}
     \inprod{\beta}{\Reg(t)} \in [\rho_1, \rho_2], \,\, C\Param \Reg(t) \leq d, \text{ and} \\
    \kappa(\Param \Reg(t)) \leq \inprod{\beta}{\Reg(t)} \,\, \text{for all}\,\, t\in [0,\horizon].
\end{cases}
\end{align}
The piecewise constant generating functions \((\reg_i(\cdot))_{i=1}^{N}\) grid the time horizon into \(N\) subintervals, ensuring that the control and slack trajectories remain constant within each subinterval. Consequently, the constraints \eqref{constr:control} become time-independent within each such \(N\) many subintervals, the constraints in \ref{constr:control} can be equivalently expressed as
\begin{align}
\label{constr:time_indp}
    \begin{cases}
    \inprod{\beta}{e_k} \in [\rho_1, \rho_2], \,\,C\Param e_k \leq d, \text{ and }\\ 
    \kappa\left(\Param e_k \right) \leq \inprod{\beta}{e_k} \text{ for all } k \in [1;N],
    \end{cases}
\end{align}
where \((e_k)_{k=1}^{N} \subset \Rbb^N\) are the standard canonical basis for \(\Rbb^N\). In a similar fashion, exploiting the piecewise constant nature of the generating functions, we rewrite the parametrized cost function as
\begin{align}
    \label{eq:param_value_func}
\hspace{-2mm}\objectivedisc\bigl(\param,\Param \Reg(\cdot), \inprod{\beta}{\Reg(\cdot)}\bigr) &= \hspace{-1mm} 
    \fcost(\st(\horizon)) + \hspace{-0.5mm}\sum_{i=1}^{N} \hspace{-1mm}\frac{\horizon\rcost\bigl(\inprod{\beta}{e_i}\bigr)}{N}.
\end{align}
For notational simplicity, we define the set of admissible parameters by
\begin{align}\label{eq:control:feas:set}
     \Paramset \Let  \left\{(\Param, \beta) \middle\vert  
        \text{ constraints in \eqref{constr:time_indp} holds}
        \right\}.
\end{align}
Then \eqref{e:fOCP} can be equivalently recast as the following CSIP:
\begin{equation}
    \label{e:r_SIP}
\begin{aligned}
    &\inf_{(\Param, \beta) \in \Rbb^{\dimcon \times N} \times \Rbb^N}&& \objectivedisc\bigl(\param,\Param \Reg(\cdot), \inprod{\beta}{\Reg(\cdot)}\bigr)\\
    &  \sbjto	&&  
    \text{constraints \eqref{eq:state_constraint_a}, \eqref{eq:state_constraint_b}} \text{ and } (\Param, \beta) \in \Paramset.
\end{aligned}
\end{equation}
Our primary focus is on solving problem \eqref{e:r_SIP}.\footnote{Recall that this problem includes the requirement that the path constraint \(\st(t) \in \admst\) must hold for uncountably many \(t \in \lcrc{0}{\horizon}\).} Before proceeding towards one of our main results, we show that the feasible set corresponding to the CSIP \eqref{e:r_SIP} admits a nice structure. 

\subsection{Technical results}\label{subsec:tech:results}
The next lemma contains some qualitative properties of the CSIP \eqref{e:r_SIP} and its feasible set, which will be useful for our subsequent analysis.
\begin{thm}\label{lemm:feas:conv:comp}
Fix \(\param \in \fsblset\) and consider the OCP \eqref{e:fOCP} and the finitely parametrized CSIP \eqref{e:r_SIP}. The feasible set for the CSIP \eqref{e:r_SIP}, defined by
\begin{align}
  \feasSip \Let  \left\{\hspace{-1mm}(\Param, \beta) \middle\vert \;
\begin{array}{@{}l@{}}
       \st(0) = \param,\, r_F(\st(\horizon)) = 0,\,
        \inprod{\beta}{e_k} \in [\rho_1, \rho_2], \\ C\Param e_k \leq d,\, 
    \kappa\left(\Param e_k \right) \leq \inprod{\beta}{e_k} \text{ for all }\\k \in [1;N] \text{ and }x(t) \in \admst \text{ for all }t \in \lcrc{0}{\horizon}
        \end{array}
        \right\},\nn
\end{align}
is compact and convex. Moreover, the CSIP \eqref{e:r_SIP} admits a feasible solution over \(\feasSip\). 
\end{thm}
\begin{pf}
Fix \(t \in \lcrc{0}{\horizon}\) and let
\begin{align}
  \feasSip^{t} \Let  \left\{\hspace{-1mm}(\Param, \beta) \middle\vert \;
\begin{array}{@{}l@{}}
       \st(0) = \param,\, r_F(\st(\horizon)) = 0,\,
        \inprod{\beta}{e_k} \in [\rho_1, \rho_2], \\ C\Param e_k \leq d,\, 
    \kappa\left(\Param e_k \right) \leq \inprod{\beta}{e_k} \text{ for all }\\k \in [1;N] \text{ and }x(t) \in \admst 
        \end{array}
        \right\}.\nn
\end{align}
We make the following observations concerning the constraints in \(\feasSip^{t}\):
\begin{enumerate}[label=\textup{L-\alph*)}, leftmargin=*, widest=b, align=left]
\item \label{lemm:feas:1} The set defined by 
\[\feasSip_1 \Let \aset[\big]{(\Param,\beta) \suchthat r_F\bigl(\st(\horizon;\param,\Param)\bigr) = 0}.\]
is closed and convex. Indeed, recall from \eqref{eq:p_sol} that 
\begin{align}
\st(T;\param,\Param) &=  e^{AT}\param +  \int_{0}^{T} \epower{A(T-\tau)} \left(B \Param \Reg(\tau) + \distfield \right)  \odif{\tau} \nn \\ & = C_0 + \mathcal{L}(\Param),\nn
\end{align}
where \(C_0\) is a constant and \(\Param \mapsto \mathcal{L}(\Param)\Let  \int_{0}^T \epower{A(T-\tau)} B \Param \Reg(\tau) \odif{\tau}\) is affine. 
Convexity of \(\feasSip_1\) follows readily from the fact that \(\Param \mapsto r_F \circ \st(\horizon;\param,\Param)\) is affine since \(r_F(\cdot)\) is.

\item \label{lemm:feas:2} Define the set 
\[\feasSip_2 \Let \aset[\big]{(\Param,\beta) \suchthat \kappa(\Param e_k) \le \inprod{\beta}{e_k} \text{ for }k \in [1;N]}.\]
For each \(k \in [1;N]\) consider the set \(\feasSip_2^{k} \Let \aset[]{(\Param,\beta) \suchthat \kappa(\Param e_k) \le \inprod{\beta}{e_k}}\) and notice that 
\begin{align}
   \feasSip_2 = \bigcap_{k=1}^{N} \feasSip_2^{k} = \bigcap_{k=1}^N \aset[\big]{(\Param,\beta) \suchthat f_k(\Param,\beta) \le 0},\nn
\end{align}
where \((\Param,\beta) \mapsto f_k(\Param,\beta) \Let \kappa(\Param e_k) - \inprod{\beta}{e_k}.\) Due to continuity of \(\Param \mapsto \kappa \circ \Param e_k\) and \(\beta \mapsto \inprod{\beta}{e_k}\) the continuity of \(f_k(\cdot,\cdot)\) follows. Thus \( \feasSip_2  = \bigcap_{k=1}^N f_k^{-1}\bigl(\lorc{-\infty}{0}\bigr)\) is closed. Convexity of \( \feasSip_2\) follows immediately.

\item \label{lemm:feas:3} Recall from the problem data \ref{OCPdata1} that \(C \in \Rbb^{m \times \dimcon} \), \( d \in \Rbb^m \), and define the set 
\[\feasSip_3 \Let \aset[\big]{(\Param,\beta) \suchthat C\Param e_k \le d \text{ for }k \in [1;N]},\]
and observe that it is the intersection of finitely many closed and convex half spaces, i.e.,
\begin{align}
    \feasSip_3 = \bigcap_{k=1}^{N}\bigcap_{i=1}^m \aset[\big]{(\Param,\beta) \suchthat [C\Param e_k]_i \le d_i}\nn
\end{align}
and thus it is closed and convex. 

\item \label{lemm:feas:4} Given that \(\admst\) is linear subspace of the \(\Rbb^{d_x}\) and \(\Param \mapsto \st(t;\param,\Param)\) is affine, closedness and convexity of
\begin{align}
   \feasSip_4^t \Let \aset[\big]{(\Param,\beta) \suchthat \st(t;\param,\Param) \in \admst} = \st(t;\param,\cdot)^{-1}\{\admst\} \nn
\end{align}
is immediate. 

\item \label{lemm:feas:5} \( \feasSip_5 \Let \aset[\big]{(\Param,\beta) \suchthat \rho_1 \le \inprod{\beta}{e_k} \le \rho_2 \text{ for }k\in [1;N]}\) is closed and convex by definition. 
\end{enumerate}
\vspace{1mm}
Thus, the set \(\feasSip^t\) being the intersection of the sets \(\feasSip_1,\feasSip_2,\feasSip_3,\feasSip^t_4 \,(\text{the only }t \text{-dependent set}),\) and \(\feasSip_5\) is closed and convex. 

Note that for \(k\in [1;N]\), \(\kappa(\Param e_k) \le \inprod{\beta}{e_k} \in \lcrc{\rho_1}{\rho_2}\) implies, along with the near-monotone property of \(\kappa(\cdot)\), that \(\Param e_k\) takes values in a bounded set which enforces \(\Param\) be to bounded. Moreover, the constraint \(\inprod{\beta}{e_k} \in \lcrc{\rho_1}{\rho_2}\) implies that \(\norm{\beta} \le \ol{C} \max_{k} \abs{\inprod{\beta}{e_k}} \le \ol{C}\rho_2\) for some constant \(\ol{C}>0\); so \(\beta\) is bounded. The other constraints are either affine or linear in the \((\Param,\beta)\) pair and they cannot push \((\Param,\beta)\) from \(\aset[\big]{ \Param \suchthat \kappa(\Param e_k) \le \rho_2} \times \aset[\big]{\beta \suchthat \inprod{\beta}{e_k} \in \lcrc{\rho_1}{\rho_2}}\). Consequently, the feasible set \(\feasSip = \bigcap_{t \in \lcrc{0}{\horizon}}\feasSip^t\) is closed, bounded (hence compact), and convex. The proof of the first assertion is complete.

\begin{enumerate}[label=\textup{L-\alph*)}, leftmargin=*, widest=ii, start=6, align=left]
\item \label{lemm:feas:6} (Convexity and continuity of the objective function in \eqref{e:r_SIP}) Since \(\Param \mapsto \st(\horizon;\param,\Param)\) is affine and \(\fcost(\cdot)\) is affine by hypothesis, the mapping \(\feasSip \ni (\Param,\beta) \mapsto \fcost \circ \st(\horizon;\param,\Param)\) is affine and continuous. Moreover, due to the convexity of the map \(\beta \mapsto \inprod{\beta}{e_i}\) for all \(i=1,\ldots,N\) and convexity of \(\rcost(\cdot)\), the mapping \(\feasSip \ni (\Param,\beta) \mapsto \rcost\bigl(\inprod{\beta}{e_i}\bigr)\) also is convex and continuous. 
\end{enumerate}
\vspace{1mm}
The second assertion follows readily from the Weierstrass theorem \cite[Theorem 2.30]{ref:beck2014introduction}. \(\qedwhite\)
\end{pf}

We now step towards our algorithmic developments for solving the CSIP \eqref{e:r_SIP}. To this end, we introduce a relaxed formulation of \eqref{e:r_SIP}. Define
\begin{align}\label{dvar}
\dvar \Let \dimcon N + N = \text{dim. of decision space of } \eqref{e:r_SIP}.
\end{align}
For every fixed parameter \(\param \in \fsblset\), define the map \[[0,\horizon]^{\dvar} \ni \bigl(t^1,\ldots,t^{\dvar}\bigr) \teL \tseq \mapsto \gfunc(\tseq; \param) \in \Rbb\] 
defined by
\begin{align}
\label{eq:g_func}
\gfunc(\tseq; \param) \Let 
\inf_{(\Param, \beta)} \quad 
&\objectivedisc\bigl(\param,\Param \Reg(\cdot), \inprod{\beta}{\Reg(\cdot)}\bigr) \nn\\
\sbjto \quad
&\begin{cases}
\text{constraints \eqref{eq:state_constraint_a}},\;\; \st(t^i;\param,\Param) \in \admst \\
\text{for all } i \in [1;\dvar],\; (\Param, \beta) \in \Paramset.
\end{cases}
\end{align}
Recall that \(\Paramset\) in \eqref{eq:control:feas:set} involves finitely many constraints in \eqref{eq:g_func}, whereas the path constraints are required to be satisfied at finitely many instants \((t^1,\ldots,t^{\dvar}) \in \lcrc{0}{\horizon}^{\dvar}\) instead of the uncountable family which appears in \eqref{e:r_SIP}. The following is our key result, which establishes a numerically viable method for the lossless extraction of the optimal value of the CSIP \eqref{e:r_SIP} by involving \eqref{eq:g_func} in the following way.
\begin{thm}
\label{thm:sampling_thm}
Consider the optimal control problem \eqref{e:fOCP}, the CSIP \eqref{e:r_SIP}, and the convex program \eqref{eq:g_func} with their corresponding notations and data \ref{OCPdata1}-\ref{OCPdata3}. Suppose that Assumption \ref{assum:slater_cond} is satisfied. Consider the maximization problem
\begin{align}\label{eq:sup_sip}
\sup_{\tseq \in [0,T]^{\dvar}} \gfunc(\tseq;\param) \quad \text{for }\param \in \admst_{\horizon}.
\end{align}
Then:
\begin{enumerate}[label=(\roman*)]
    \item (Regularity of \(\gfunc(\cdot;\param)\)) \label{sup_sup_1} \( \lcrc{0}{\horizon}^{\dvar} \ni \bigl(t^1,\ldots,t^{\dvar}\bigr) \teL \tseq \mapsto \gfunc(\tseq; \param) \in \Rbb\) is continuous for every \(\param \in \fsblset\),
    \item (Existence of optimizers) \label{sup_sup_2} there exists an optimizer \(\tseq\as(\param)\) that solves \eqref{eq:sup_sip}, and
    \item \label{sup_sup_3} (Exact solution) \(\valf(\param) = \gfunc(\tseq^{\ast}(\param);\param),\) where \(\valf(\cdot)\) is defined in \eqref{eq:value funcion}.
\end{enumerate}
\end{thm}
\begin{rem}
\label{rem:on_reformulation}
The assertions in Theorem \ref{thm:sampling_thm} are of two types --- qualitative and quantitative. \ref{sup_sup_1} and \ref{sup_sup_2} are qualitative, and are concerned with the regularity of \(\gfunc(\cdot;\param)\) that are useful for the selection of numerical algorithms. \ref{sup_sup_3} provides the quantitative angle --- equality of values of \eqref{e:fOCP} and \eqref{eq:sup_sip}. More precisely:

\begin{itemize}[leftmargin=*, label = \(\circ\)]
     \item Assertion \ref{sup_sup_1} establishes the continuity of the mapping \(\gfunc(\cdot;\param)\) for \(\param \in \fsblset\), which is important for good numerical performance of global optimization algorithms that may be employed to solve \eqref{eq:sup_sip}; see the algorithmic architecture \(\glbopt\) in \S\ref{subsec:algo} ahead. Under additional smoothness assumptions on the constraint sets (such as sublevel sets of smooth functions or polytopes), it can be shown that \(\gfunc(\cdot;\param)\) is Lipschitz continuous \cite{ref:FiaIsh-90}, which enables an array of global optimization algorithms especially designed for the class of Lipschtiz functions; see  \S\ref{subsec:algo}. 
     
    \item The convex program \eqref{eq:g_func} is a relaxed version of \eqref{e:r_SIP}, having finitely many constraints on the paths; assertion \ref{sup_sup_2} in Theorem \ref{thm:sampling_thm} asserts that it suffices to enforce the path constraints at a carefully selected \(\dvar\)-tuple \(\tseq\) to obtain an \emph{exact} solution of the CSIP \eqref{e:r_SIP}, i.e., \(\valf(\param) = \gfunc(\tseq^{\ast};\param)\). To determine the optimal \(\dvar\)-tuple, the maximization problem \eqref{eq:sup_sip} in Theorem \ref{thm:sampling_thm} must be (globally) solved over \(\lcrc{0}{\horizon}^{\dvar}\).
\end{itemize}
\vspace{1mm}
  Our second main result, Theorem \ref{thrm:opt:extract} will demonstrate that by employing a regularization process one can construct a minimizing sequence that comprises of solutions of suitably regularized revisions of \eqref{e:r_SIP}, that converge to an optimizer of the CSIP \eqref{e:r_SIP} as the regularization parameter \(\downarrow 0\). To the best of our knowledge, this is the only numerically viable algorithm that \emph{exactly} solves the class of OCPs \eqref{e:fOCP}, ensuring that the optimal values of the original problem and the relaxed problem \eqref{eq:g_func} coincide, with the optimizers of the relaxed problem also solving the original CSIP.
\end{rem}


\begin{rem}[On exactness and conservatism]
Note that, precisely, there are two layers: 
\begin{itemize}[label= \(\circ\), leftmargin=*]
\item  \textbf{Layer one} is a modeling layer, where we choose a finite dictionary and thereby restrict the admissible spaces of control and slack. This is a standard route in numerical optimal control and the parametrization can be refined by increasing \(N\) if higher fidelity is needed.
\item \textbf{Layer two} is a robust optimization layer, where, given this finite parametrization, we address the semi-infinite robust constraints \emph{exactly} via the CSIP formulation and Theorems \ref{thm:sampling_thm} and \ref{thrm:opt:extract} (ahead). The term, ``exact'' in this layer means: the established CSIP-based procedure recovers the original optimal value (Theorem \ref{thm:sampling_thm}) and the optimizer of \eqref{e:r_SIP} (Theorem \ref{thrm:opt:extract}).
\end{itemize}
All our claims of \emph{no conservatism} and \emph{exactness} always refer to the second indicated layer. The claim of exactness means that, for the infinitely constrained (but finitely parameterized) OCP \eqref{e:r_SIP} arising from the first layer, both the optimal value and the optimizers coincide with those of the finitely constrained convex program \eqref{eq:sup_sip}.
\end{rem}

\begin{pf}[Proof of Theorem \ref{thm:sampling_thm}]
Fix \(\param \in \fsblset\); we show the continuity of the mapping \(\lcrc{0}{\horizon}^{\dvar} \ni \bigl(t^1,\ldots,t^{\dvar}\bigr) \teL \tseq \mapsto \gfunc(\tseq; \param) \in \Rbb\) employing results from variational analysis \cite[Chapter 3]{ref:DonRoc-14}. 

Fix \(\param \in \fsblset\) and \((\ol{\Param},\ol{\beta},\ol{\tseq}) \in \feasSip \times \lcrc{0}{\horizon}^{\dvar}\). Observe that: 
\begin{enumerate}[label=\textup{T-\alph*)}, leftmargin=*, widest=ii, align=left]

\item \label{thrm:point:a} The constraint mappings in the feasible set \(\feasSip\), i.e., the maps  \((\Param,\beta)\mapsto r_F(\st(\horizon;\param,\Param))\), \((\Param,\beta)\mapsto\inprod{\beta}{e_k}\), \((\Param,\beta)\mapsto C\Param e_k-d\), \((\Param,\beta)\mapsto\kappa\left(\Param e_k \right) -\inprod{\beta}{e_k}\) for all \(k \in [1;N]\), and \((\Param,\beta)\mapsto x(t^i;\param,\Param)\) for all \(i \in [1;\dvar]\) are continuous due to the problem data and their definitions.

\item \label{thrm:point:b} Moreover, the constraint mappings in \ref{thrm:point:a} are convex (or affine, depending on their images) in the decision variable \((\Param,\beta)\) for each fixed \(\tseq \in \lcrc{0}{\horizon}^{\dvar}\). See the steps \ref{lemm:feas:1} through \ref{lemm:feas:5} in the proof of Theorem \ref{lemm:feas:conv:comp}.
\end{enumerate}
\vspace{1mm}
Consequently, the feasible set mapping for the problem \eqref{eq:g_func} is continuous at \(\ol{\tseq}\) in the sense of Painlev\'e-Kuratowski \cite[Chapter 3, \S 3.2]{ref:DonRoc-14}. This follows from \cite[Chapter 3, Example 3B.4]{ref:DonRoc-14} because of the properties \ref{thrm:point:a}--\ref{thrm:point:b} along with the Slater-like condition in Assumption \ref{assum:slater_cond}. To see the continuity of \(\gfunc(\cdot;\param)\) for a fixed \(\param \in \fsblset\), we will employ the regularity result \cite[Chapter 3, Theorem 3B.5]{ref:DonRoc-14}. To this end,
\begin{enumerate}[label=\textup{T-\alph*)}, leftmargin=*, widest=ii, start=3, align=left]

\item \label{thrm:point:c} recall that \(\param \in \fsblset\) and \((\ol{\Param},\ol{\beta},\ol{\tseq}) \in \feasSip \times \lcrc{0}{\horizon}^{\dvar}\) are fixed; the feasible set of \eqref{eq:g_func} is nonempty and bounded via Theorem \ref{lemm:feas:conv:comp};

\item \label{thrm:point:d} the cost function \(\feasSip \times \lcrc{0}{\horizon}^{\dvar} \ni (\ol{\Param},\ol{\beta},\ol{\tseq}) \mapsto \objectivedisc(\param,\Param\Reg(\cdot),\inprod{\beta}{\Reg(\cdot)})\) in \eqref{eq:g_func} is continuous;

\item \label{thrm:point:e} the feasible set for \eqref{eq:g_func} is continuous as a set-valued map as shown above, and thus, it is Pompeiu-Hausdorff continuous at \(\ol{\tseq}\in \lcrc{0}{\horizon}^{\dvar}\). 
\end{enumerate}
\vspace{1mm}
The hypotheses of \cite[Chapter 3, Theorem 3B.5]{ref:DonRoc-14} are, therefore, verified, which gives us the continuity of the value function mapping \(\tseq \mapsto \gfunc(\tseq;\param)\) for fixed \(\param \in \fsblset\) over the domain \(\lcrc{0}{\horizon}^{\dvar}\). Thus, part-\ref{sup_sup_1} of Theorem \ref{thm:sampling_thm} is established. 

Since \(\lcrc{0}{\horizon}^{\dvar}\) is compact, by Weierstrass's Theorem \cite[Theorem 2.30]{ref:beck2014introduction}, assertion \ref{sup_sup_2} follows.

We now establish assertion \ref{sup_sup_3}, for which all the necessary ingredients have been prepared in the preceding analysis. For the CSIP \eqref{e:r_SIP}, observe that: 
\begin{enumerate}[label=\textup{T-\alph*)}, leftmargin=*, widest=ii, start=6, align=left]
\item \label{thrm:csip:cond:1} The domain \(\feasSip \subset \Rbb^{\dimcon \times N} \times \Rbb^{N}\) of the CSIP \eqref{e:r_SIP} is closed and convex with nonempty interior; see Theorem \ref{lemm:feas:conv:comp}. 
\item \label{thrm:csip:cond:2} From Assumption \ref{assum:slater_cond} we have that the CSIP \eqref{e:r_SIP} is strictly feasible and the admissible set has a nonempty interior. 
\item \label{thrm:csip:cond:3} The objective function is convex and continuous; see point \ref{lemm:feas:6} of Theorem \ref{lemm:feas:conv:comp}. 
\item \label{thrm:csip:cond:4} The constraint mappings are continuous in \((\Param,\beta, t) \in \feasSip \times \lcrc{0}{\horizon}\) and convex in the decision variables \((\Param,\beta) \in \feasSip\) for each fixed \(t \in \lcrc{0}{\horizon}\). This has already been established in Theorem \ref{lemm:feas:conv:comp} and \ref{thrm:point:a}--\ref{thrm:point:b}. 
\item \label{thrm:csip:cond:5} The constraint index set \(\lcrc{0}{\horizon}\), i.e., the set of semi-infinite variables is compact. 
\end{enumerate}
Thus, the problem data ((1.1)-a)--((1.1)-e) in \cite{ref:DasAraCheCha-22} are all satisfied. From \cite[Theorem 1]{ref:DasAraCheCha-22} it follows that for each \(\param \in \fsblset\), \(\valuefunc(\param) = \gfunc(\tseq^{\ast}(\param);\param)\). \(\qedwhite\)
\end{pf}

Our next result establishes a way to extract the optimizers of the SIP \eqref{e:r_SIP} via solving a sequence of regularized problems. Let \(\eta: \Rbb^{\dimcon \times N} \times \Rbb^N \rightarrow \Rbb\) be a continuous, positive, and strictly convex function. Consider the CSIP \eqref{e:r_SIP} with the objective 
\[\objective^{\eps}\Let \objectivedisc\bigl(\param,\Param \Reg(\cdot), \inprod{\beta}{\Reg(\cdot)}\bigr) + \eps \eta(\Param,\beta),\]
with the same set of constraints and \(\eps>0\) is a regularization parameter.\footnote{Of course \(\objective^{\eps}\) depends on \(\param, \eps,\Param \Reg(\cdot)\), and \(\inprod{\beta}{\Reg(\cdot)}\), which we suppress for a cleaner notation.} For every fixed parameter \(\param \in \fsblset\) and \(\eps>0\), define the map \(\lcrc{0}{\horizon}^{\dvar} \ni \bigl(t^1,\ldots,t^{\dvar}\bigr) \teL \tseq \mapsto \gfunc(\tseq;\param,\eps) \in \Rbb\) given by
    \begin{align}\label{eq:g_func_regu}
        \gfunc(\tseq; \param,\eps) \Let \inf_{(\Param, \beta)}	\quad &\objective^{\eps}\nn\\
        \sbjto \quad & \begin{cases}
        \text{constraints \eqref{eq:state_constraint_a}},\,\, \st(t^i;\param,\Param) \in \admst\\
        \text{for all }\, i\in [1;\dvar] \text{ and } (\Param, \beta) \in \Paramset.
        \end{cases}
\end{align}
The following is our second main result. 
\begin{thm}\label{thrm:opt:extract}
Let \(\eta: \Rbb^{\dimcon \times N} \times \Rbb^N \rightarrow \Rbb\) be as defined above and suppose that \(c(\cdot)\) is near-monotone. We have the following assertions: (a) fix \(\param \in \fsblset\) and \(\eps>0\); the mapping \(\lcrc{0}{\horizon}^{\dvar} \ni \tseq \mapsto \gfunc(\tseq;\param,\eps)\) is continuous and there exists an optimizer \(\tseq\as(\param,\eps)\) which solves an appropriately modified version of the global optimization problem \eqref{eq:sup_sip}. (b) for \(\param \in \fsblset\) and \(\eps>0\) we have \(\valf(\param,\eps) = \gfunc(\tseq\as(\param,\eps);\param,\eps)\). (c) for a fixed \(\param\in \fsblset\) and \(\eps>0\), let \(\valf(\param,\eps)\) be the value of the CSIP \eqref{e:r_SIP} with the regularization (i.e., with the objective \(\objective^{\eps}\)). Then, the sequence of (unique) optimizers \(\bigl(\Param\as(\param,\eps),\beta\as(\param,\eps)\bigr)\) defined by
\begin{align*}
(\Param\as(\param,\eps),\beta\as(\param,\eps)) \Let  \argmin_{(\Param,\beta)} \quad &  \objective^{\eps}\\
        \sbjto \quad & \begin{cases}\text{the constraints in \eqref{e:r_SIP}}\\ 
        \text{hold for }\tseq^{\ast}(\param,\eps),
        \end{cases}
\end{align*}
converges to an optimizer of \eqref{e:r_SIP} as \(\eps \downarrow 0\).
\end{thm}
\begin{pf}
As the feasible set remains unchanged and the addition of the term \((\Param,\beta) \mapsto \eta(\Param,\beta)\) preserves the convexity of the objective in \((\Param,\beta)\), parts (a) and (b) follow analogously to the proof of Theorem \ref{thm:sampling_thm}. We provide a brief proof for part (c).\footnote{The existence of solutions and properties established in Theorem \ref{lemm:feas:conv:comp} extend similarly to the feasible sets of the regularized CSIP with objective \(\objective^{\eps}\) and the convex program \eqref{eq:g_func_regu}.} Note that \((\alpha,\beta) \mapsto \objectivedisc(\param, \Param\Reg(\cdot), \inprod{\beta}{\Reg(\cdot)})\) is convex and \(\eta(\cdot,\cdot)\) is strictly convex. Moreover, due to \(c(\cdot)\) being near-monotone, \(\objectivedisc(\cdot)\) is near-monotone. The conditions of \cite[Proposition 3.3]{ref:PPDC-23} are met, allowing us to directly apply its result to establish Theorem \ref{thrm:opt:extract}. The proof is complete.  \(\qedwhite\)
\end{pf}

\subsection{Architecture to numerically solve \eqref{e:r_SIP}}\label{subsec:algo}
We present an algorithmic architecture --- \(\glbopt\) --- to solve the global optimization problem \eqref{eq:sup_sip}. The framework \(\glbopt\) provides a unified approach for synthesizing constrained control actions by selecting a suitable global optimization routine to solve the maximization problem \eqref{eq:sup_sip}.\footnote{In our architecture, the main computational burden arises from this outer global maximization step. As a result, pursuing the strongest ``exactness'' guarantee comes at a tangible runtime cost relative to approaches that rely on local/gradient-based oracles. To illustrate this trade-off, in \S\ref{sec:num_exp} we also report results with the gradient-based module \(\textsf{gradOL}\) \cite{ref:GradOL}, which replaces the global maximization by a gradient-driven search and can therefore be substantially faster, albeit without the same global certificate. For this reason, the full \(\glbopt\) architecture is currently best suited to offline trajectory generation and certification, rather than online deployment.} The choice of the routine depends on the properties of the function \(\gfunc(\cdot;\param,\eps)\).

\begin{itemize}[leftmargin=*,label=\(\circ\)]
    \item If \(\gfunc(\cdot;\param,\eps)\) is continuous, as shown in Theorem \ref{thm:sampling_thm}, convergence is ensured by using the simulated annealing oracle or differential evolution oracle within the \(\glbopt\) architecture, provided an appropriate cooling schedule and mild conditions on the transition kernel are met, as established in \cite[Theorem 2]{ref:DasAraCheCha-22}, \cite[Theorem 1]{ref:CJPB-92}, and \cite{ref:storn1997differential}.

    \item If \(\gfunc(\cdot;\param,\eps)\) is Lipschitz continuous, global optimization oracles like \texttt{SequOOL} (Sequential Optimistic Optimization with Levels) or \texttt{LIPO} (Lipschitz Optimization based on Local Partitions) can be utilized. \texttt{SequOOL} leverages a deterministic hierarchical approach \cite{ref:BarGabVal-19}, achieving an exponential regret bound \cite{ref:GriValMun-15}; \texttt{LIPO} uses stochastic averaging \cite{ref:MalVay-17}, delivering robust PAC-style regret bounds. Both methods, when paired with randomized restarts, lead to rapid convergence to global optima in practice.

    \item Several criteria in \(\glbopt\) can be picked depending upon the type of applications. \(\textsf{S1}\) is the stopping criterion for the outer loop, determining when to return an optimizer estimate. It can be defined, for example, as the norm of the difference between successive estimates:
\(\textsf{S1}(j) \Let \norm{(\ol{\Param}^j_{\param,\eps},\ol{\beta}^j_{\param,\eps})-(\ol{\Param}^{j-1}_{\param,\eps},\ol{\beta}^{j-1}_{\param,\eps})}.\) \(\textsf{Regu}(\cdot,\cdot)\) is the regularization update function, producing a decreasing sequence of regularizers with every update in iterations, e.g., \(\textsf{Regu}(\eps,j) \Let \frac{\eps}{2}\). The inner loop uses a global optimization routine with a stopping criterion \(\textsf{S2}\), such as a maximum number of iterations, a temperature threshold, or a bound on successive value differences (e.g., in simulated annealing). \(\textsf{S3}\) is the Improvement/Update Criterion of the global optimization routine, determining whether to accept a candidate solution or not. For example, if one elects to employ the simulated annealing global optimization routine, a candidate is accepted if it improves the current best, or otherwise with probability \(\exp{\left(\frac{\gfunc^{j,k}-\gfunc^{j}_{\max}}{T_m}\right)}\), where \(T_m\) is the temperature. We redirect the readers to \S \ref{subsec:comparision} for some numerical statistics related to an energy optimal landing problem \eqref{eq:energy_des_active}, where we include three different global optimization routines.
\end{itemize}

{
\renewcommand{\algorithmcfname}{\(\glbopt\)}
\renewcommand{\thealgocf}{}
\begin{algorithm2e}[!ht]
	\SetAlgoLined
  	\DontPrintSemicolon
\SetKwInOut{ini}{Initialize}
    \SetKwInOut{giv}{Data}
    \giv{Stopping criterion \(\textsf{S1}(\cdot)\) for extracting the optimizer and its threshold \(\tau_1\),  Stopping criterion \(\textsf{S2}(\cdot)\) for global optimization oracle and its threshold \(\tau_2\), Selection criterion for global optimization: $\textsf{S3}(\cdot)$, regularization parameter \(\eps\) initialization, regularizer function $\regul(\cdot,\cdot)$;}
    \ini{initialize constraint indices -- \(\tseq_0 \in \lcrc{0}{\horizon}^{\dvar}\), \(j=0\);\\
    select \(\gfunc_{\max} = \gfunc(\tseq_0;\param,\eps)\);\\ 
    pick the initial solution \(\bigl(\ol{\Param}_{\param,\eps},\ol{\beta}_{\param,\eps}\bigr)\)} \hspace{17.5mm}by solving:\\
     \(\argmin_{(\Param, \beta)}\left\{ \objective^{\eps}  \,\middle\vert  \begin{array}{@{}l@{}}
      \text{ constraints of } \eqref{e:r_SIP} \text{ holds at } \tseq_0 
      \end{array}
        \right\}\)
\While{$\textsf{S1}(j)\leqslant \tau_1$}{
\(\eps \gets \regul(\eps,j) \)\; 

$k\gets 0$ \;

\While{$\textsf{S2}(k)\leqslant \tau_2$}{

\emph{Evaluate:} \(\gfunc^{j,k} \Let \gfunc(\tseq_0;\param,\eps)\) as defined in \eqref{eq:g_func_regu} \;      

\emph{Recover the solution} \(\ol{\Param}^{j,k}_{\param,\eps}, \ol{\beta}^{j,k}_{\param,\eps}\) by solving: 
 \(\argmin_{(\Param, \beta)}\left\{ \objective^{\eps}  \,\middle\vert  \begin{array}{@{}l@{}}
      \text{ constraints of } \eqref{e:r_SIP} \text{ holds } \\ \text{ at } \tseq_0 
      \end{array}
        \right\}\)

Update $\gfunc_{\max}, \ol{\Param}_{\param,\eps},\ol{\beta}_{\param,\eps}, \tseq_0 \gets $
\textsf{S3}$(\gfunc_{\max}, \ol{\Param}_{\param,\eps},\ol{\beta}_{\param,\eps},\tseq_0, \gfunc^{j,k}, \ol{\Param}^{j,k}_{\xz,\eps},\ol{\beta}^{j,k}_{\xz,\eps},k)$
Update $k \gets k+1$ \;
}
Update $j\gets j+1$\;
}
\textbf{Output:} \((\Param\as(\param,\eps),\beta\as(\param,\eps))\)
\caption{A general architecture to solve \eqref{e:r_SIP}}
\label{alg:sprob}
\end{algorithm2e}
}

\begin{rem}\label{rem:on:GPU:implementations}
The algorithmic architecture \(\glbopt\) is inherently modular. At a conceptual level, the method decomposes the original problem into two distinct layers: an inner convex minimization procedure and an outer global maximization routine. In particular, the architecture \(\glbopt\) is solver-agnostic, in the sense that any suitable convex optimization algorithm and any suitable global optimization scheme can be incorporated without modification of the overarching framework. 
This structural separation yields two principal advantages. First, the performance of \(\glbopt\) automatically inherits advances in either class of solvers. Improvements in convex optimization or the development of faster global optimization solvers can be integrated directly, without redesigning the theoretical or algorithmic foundations. The framework is therefore forward-compatible and insulated from implementation-specific obsolescence. Second, the architecture is compatible with modern computational paradigms. If either the convex or the global optimization layer supports GPU acceleration, parallel execution, or distributed implementations, these computational gains propagate transparently to \(\glbopt\). GPU benefits depend heavily on the solver stack; most off-the-shelf conic/QP solvers are CPU-optimized, but the field has seen many recent developments, such as \cite{ref:GPU1,ref:GPU2,ref:GPU3}, which potentially can be utilized in our framework. These interesting computational directions will be pursued separately and will be reported in our subsequent investigations. 
\end{rem}

%% file: Numerics.tex
\section{Numerical experiments}
\label{sec:num_exp}
We explore the intricate three-DoF planetary landing problem, where the lander must precisely reach a designated touchdown point while optimizing fuel efficiency \cite[Section 7, Example 1]{ref:MWHBA-14b}. The spacecraft is under-actuated, with the thrust vector serving as the sole control input. The dynamics of the system evolve under a gravitational field, subject to constraints on the allowable landing region. The motion of the lander is governed by the following equations:
\begin{equation}\label{eq:powdes_sys}
    \begin{pmatrix}
        \dot{r}(t) \\
        \dot{v}(t)
    \end{pmatrix}
    =
    \begin{pmatrix}
        v(t) \\
        -g + u(t)
    \end{pmatrix}.
\end{equation}
Here \(t \mapsto r(t) \in \Rbb^3\), \(t \mapsto v(t) \in \Rbb^3\), and \(t \mapsto u(t) \in \Rbb^3\) represent the position, velocity, and control trajectories, respectively. The system is subject to a constant gravitational acceleration \( g = [0, 3.71, 0]^{\top} \). We consider two variants of planetary landing: \emph{energy-optimal landing} and \emph{fuel-optimal landing}, and in each case our task is to guide the vehicle from its initial state to the designated landing site while ensuring a zero final velocity. The initial and terminal conditions are specified as
\begin{equation}
\label{eq:boun_cond}
\begin{cases}
 r(0) = [400, 400, 300]^{\top} \text{ m}, \,\, v(0) = -[10, 10, 75]^{\top} \text{ m/s} \\
 r(T) = [0, 0, 0]^{\top} \text{ m}, \,\, v(T) = [0, 0, 0]^{\top} \text{ m/s}.
    \end{cases}
\end{equation}
For safety, the vehicle's descent must follow a 45-degree altitude-to-range trajectory, enforced by the constraint:
\begin{equation}
    \label{eq:planar_constraint}
    r_1(t) - r_2(t) = 0.
\end{equation}
\noindent Since the path constraints are the same as in \cite[Example 1, Section 7]{ref:MWHBA-14b}, the condition in \ref{COCP_criteria} is satisfied. 
A key distinction between these two problems lies in the objective function --- while energy-optimal landing typically involves a quadratic cost, fuel-optimal landing is defined by an \(\lpL[1]\)-norm cost, introducing additional intricacies in the optimization process. 

The numerical simulations presented here are conducted using Julia v1.10.0~\cite{ref:JBAESKBS-17}, where the internal minimization problem was solved using the \texttt{IPOPT} solver~\cite{ref:AWLTB-06} and the \texttt{JuMP} package~\cite{ref:IDJHML-2017}\footnote{The computations are performed on a PC running Ubuntu 20.04.6 LTS, equipped with 16 GB RAM and an 11th Gen Intel® Core™ i7-11800H processor @ 2.30GHz.}. For the sampling-based maximization of \(\mathcal{G}\) in \(\glbopt\), we use the \texttt{SAMIN} solver from the \texttt{Optim.jl} package~\cite{ref:PMAR-18}, which implements simulated annealing, and the \texttt{BlackBoxOptim} solver from the \texttt{Optimization.jl} package~\cite{ref:VKDCR-23}, which implements differential evolution.
Reproducible codes may be found in the GitHub repository: \url{https://github.com/v-upadhyay/pow_des_csip}.
\subsection{Energy optimal landing, case I}\label{subsec:energy_opt}
Here, the primary objective is to execute a successful landing while minimizing energy consumption, which leads to the following OCP
\begin{equation}
	\label{eq:energy_des_certain}
\begin{aligned}
&\hspace{-2mm} \inf_{\cont(\cdot)}	&& \hspace{-2mm}\int_{\tinit}^{\horizon} \norm{u(t)}^2 \odif{t} \\
&  \hspace{-2mm}\sbjto		&&  \hspace{-4mm}\begin{cases}
\dot{r}(t) = v(t),\,\, \dot{v}(t) = -g + u(t),\\
\text{boundary condition } \eqref{eq:boun_cond},\\
\text{and safety constraints } \eqref{eq:planar_constraint} \text{ for all }t\in \lcrc{0}{\horizon}, \\  
 \norm{u(t)} \in [2, 10] \text{ for all } t \in \lcrc{0}{\horizon}.\\
\end{cases}
\end{aligned}
\end{equation}
In the optimal control problem \eqref{eq:energy_des_certain}, the function \(\kappa: \Rbb^3 \lra [0, +\infty[\) is defined as \(x \mapsto \kappa(x) \Let \norm{x}\), while the function \(c: [0, +\infty[ \lra [0, +\infty[\) is defined as \(x \mapsto c(x) \Let x^2\).\footnote{Notice that the control constraint in \eqref{eq:energy_des_certain} encodes an annular admissible set.} Moreover, the mapping \(r_F: \Rbb^{\dimst} \lra \Rbb^{\dimst}\) corresponds to the identity function, and \(c_F: \Rbb^{\dimst} \lra \Rbb\) is given by \(c_F(x) \Let 0\). All these functions satisfy the problem data \ref{OCPdata1}-\ref{OCPdata3}. In OCP \eqref{eq:energy_des_certain} after employing the convexification we parametrized the control \( \cont(\cdot)\) and slack \( \slack(\cdot) \) trajectories using \eqref{e:cont_param} and \eqref{e:slack_param} keeping \(N = 100\) and \(\horizon = 22s\). As a result, the required number of time instants for maximizing \(\gfunc(\cdot;\param, 0)\) is \(\dvar = 100 \times 3 + 100 = 400,\) determined by \ref{dvar}.

The energy-optimal trajectory and the control norm obtained by applying our method are shown in Fig. \ref{fig:combined_results}: observe that the trajectory follows the specified boundary and the safety constraints, and the control norm remains within the prespecified bounds. The evolution of the error between \(s(\cdot)\) and the norm of the control trajectory is shown in Fig. \ref{fig:error_traj} (in blue). The error remains of the order of \(10^{-8}\), indicating that the solution closely satisfies the condition \(\kappa(\cont(t)) = \slack(t)\) for all \(t \in \lcrc{0}{22}\) up to standard numerical errors. 
\subsection{Fuel optimal landing}\label{subsec:fuelopt_disturbed}
Our objective here is to execute a successful landing while minimizing fuel consumption, which leads to the OCP
\begin{equation}\label{eq:powdes_certain}
\inf_{\cont(\cdot)}	\aset[\bigg]{\int_{\tinit}^{\horizon} \norm{u(t)} \odif{t} \suchthat \text{constraints in \eqref{eq:energy_des_certain} hold.}}
\end{equation}
In \eqref{eq:powdes_certain} the function \(\kappa: \Rbb^3 \lra [0, +\infty[\) is defined as \(\kappa(\cdot) \Let \norm{\cdot}\), while \(c: [0, +\infty[ \lra [0, +\infty[\) and \(r_F: \Rbb^{\dimst} \lra \Rbb^{\dimst}\) are identity functions, and \(c_F: \Rbb^{\dimst} \lra \Rbb\) is \(c_F(\cdot) \Let 0\). All these functions satisfy the problem data \ref{OCPdata1}-\ref{OCPdata3}. For \eqref{eq:powdes_certain}, we kept the same \(N = 100\) and \(\horizon = 22s\) and consequently \(\dvar = 100 \times 3 + 100 = 400\).

\begin{figure*}[!t]
    \centering
    \subfloat[]{\includegraphics[width=0.48\textwidth]{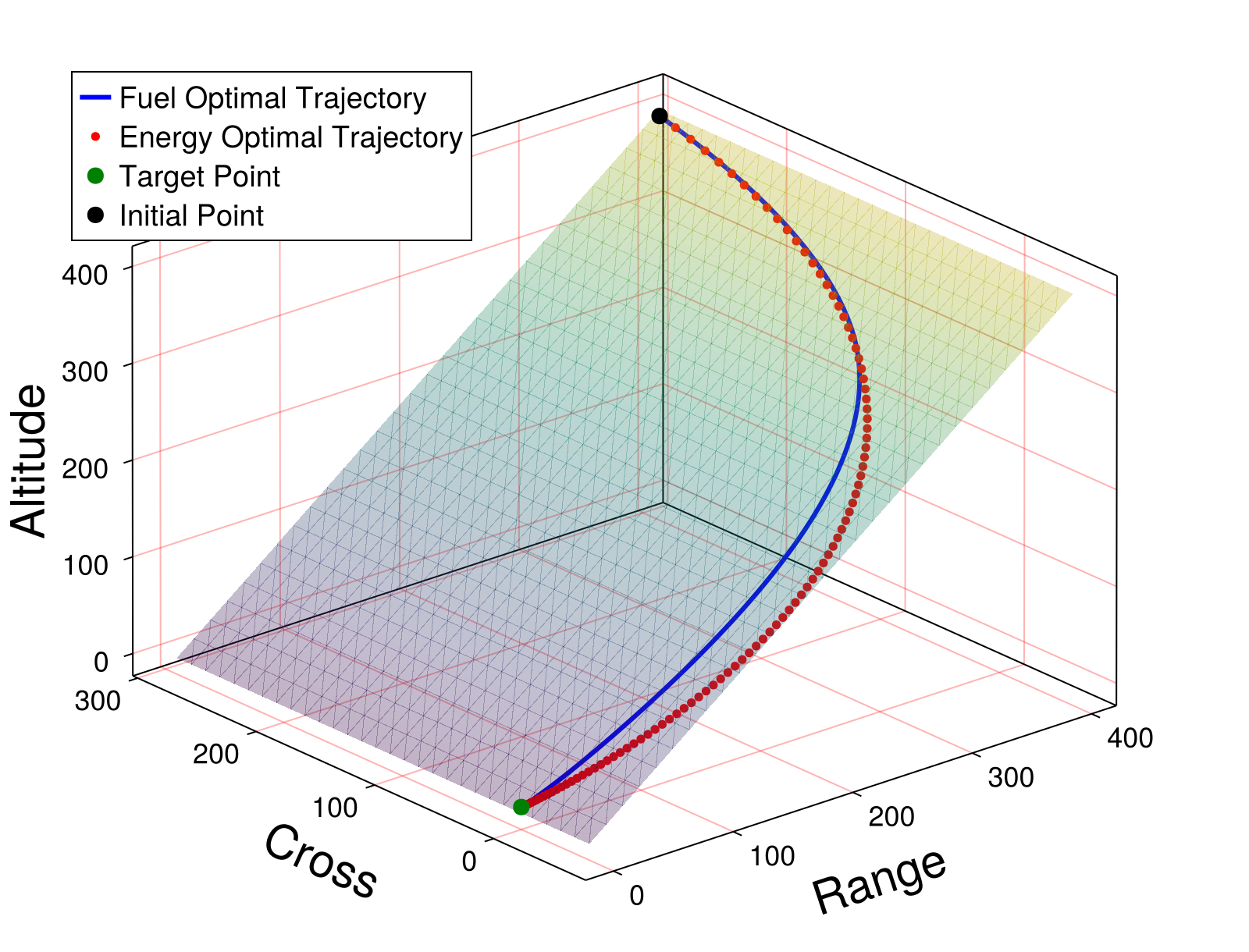}} \hfill
    \subfloat[]{\includegraphics[width=0.48\textwidth]{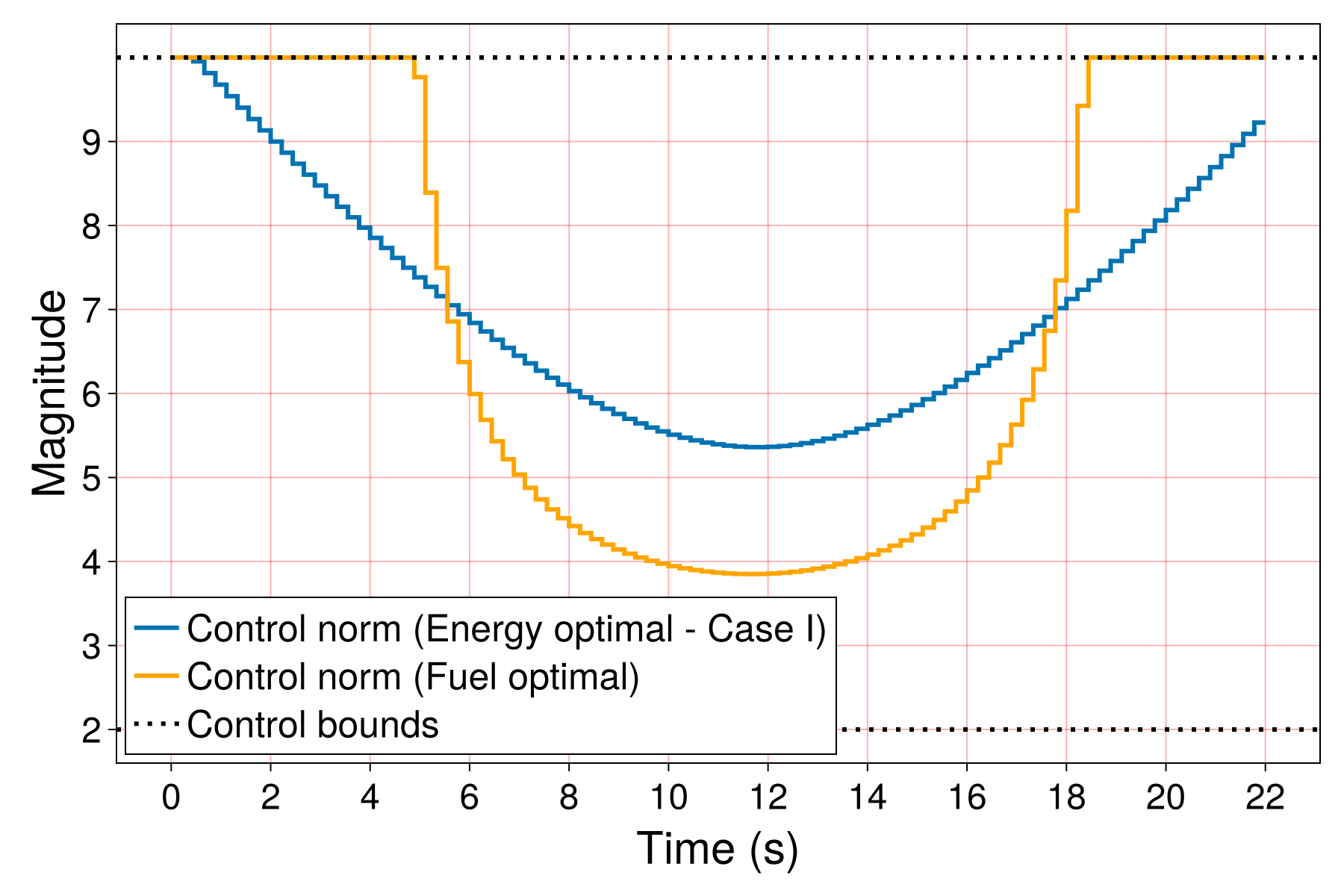}} \hfill
    \caption{(a) Fuel- and energy-optimal evolution trajectories for OCPs \eqref{eq:energy_des_certain} and \eqref{eq:powdes_certain} respectively on the subspace generated by constraint \eqref{eq:planar_constraint}. (b) Demonstration that the control norm remains within prescribed bounds and matches the slack variable.}  
    \label{fig:combined_results}
\end{figure*}

\begin{figure*}[!t]
    \centering
    \subfloat[]{\includegraphics[width=0.48\textwidth]{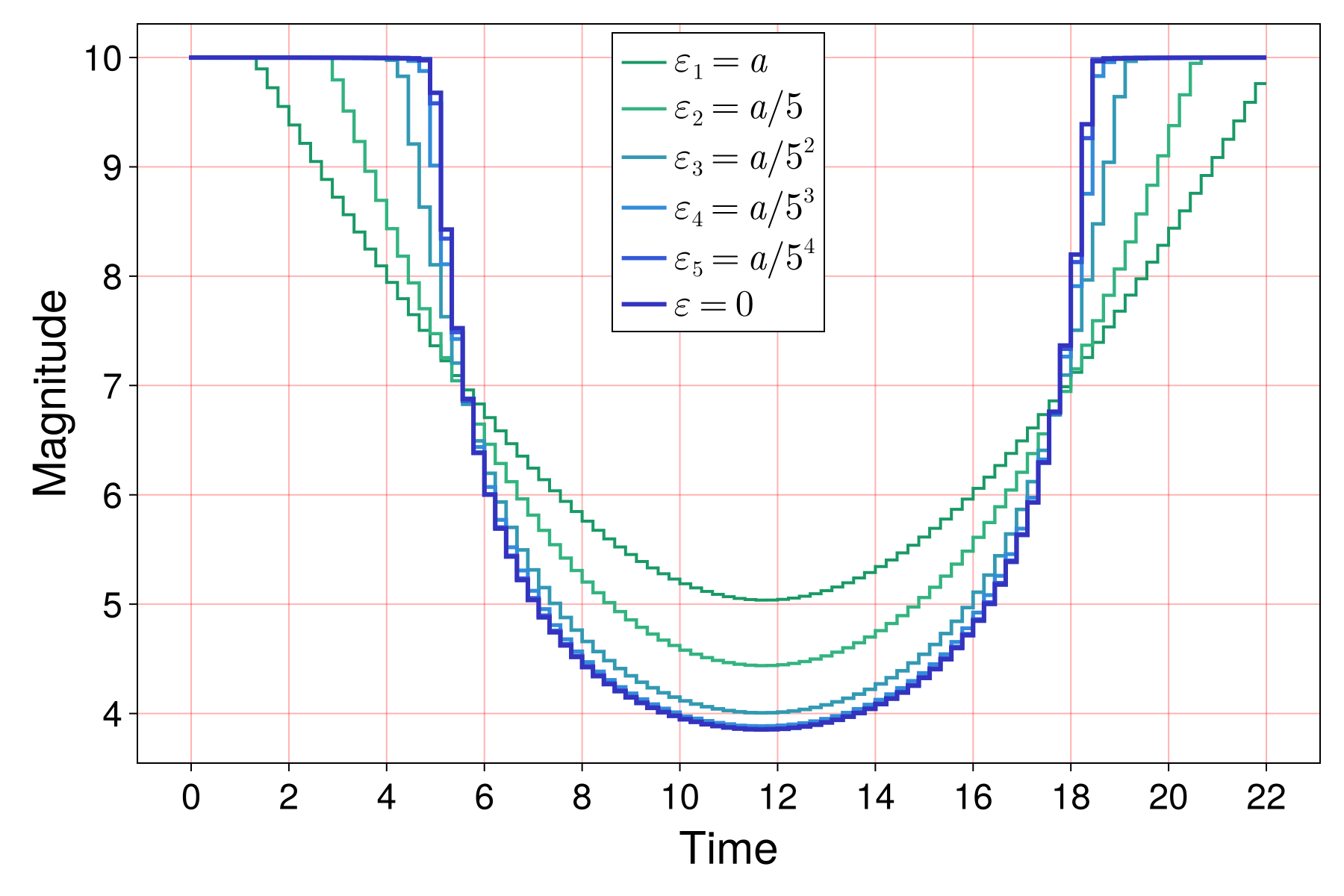} \label{fig:reg_control}} \hfill
    \subfloat[]{\includegraphics[width=0.48\textwidth]{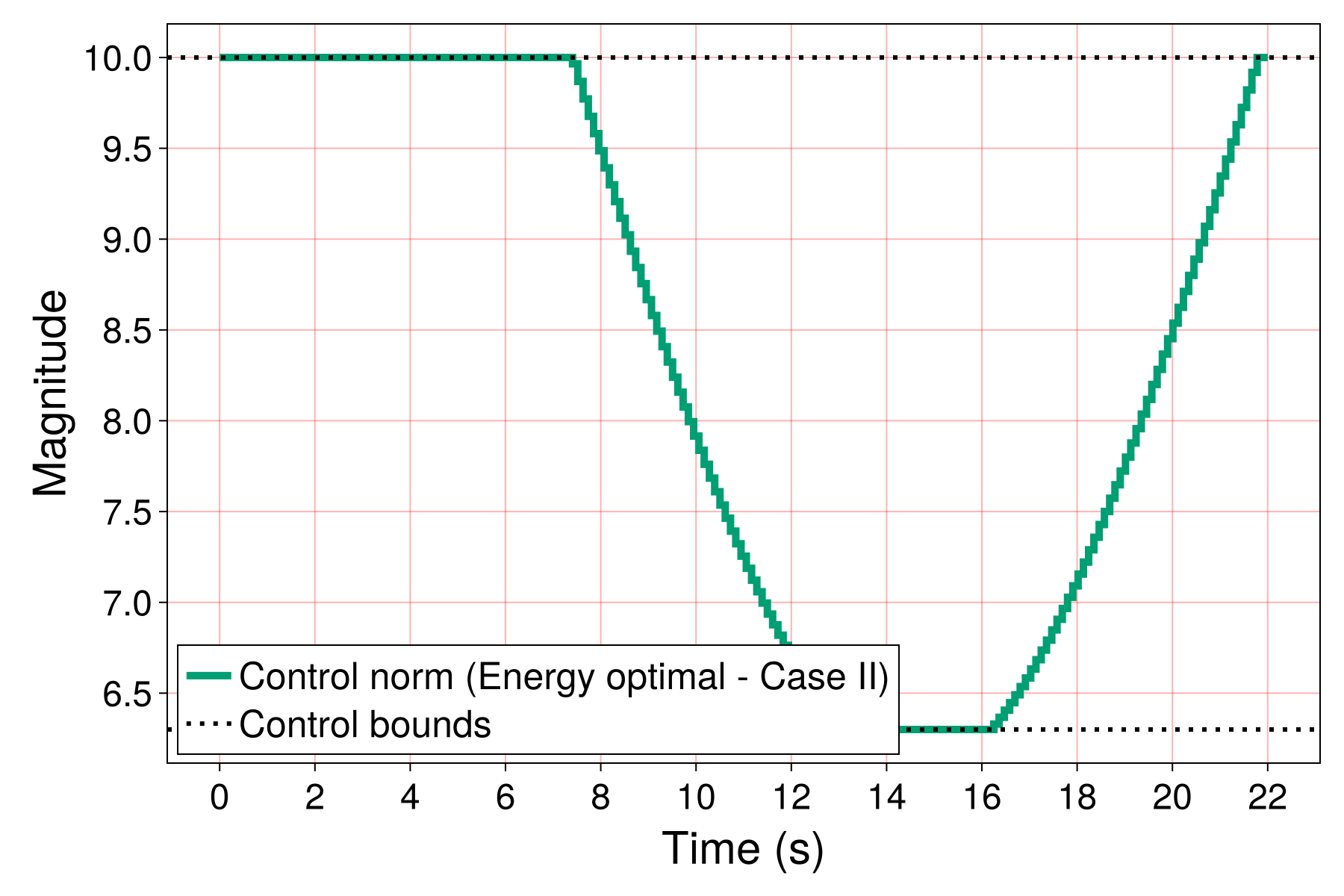} \label{fig:active_control}} \hfill
    \caption{(a) This figure illustrates the convergence of the optimal control trajectory of OCP \eqref{eq:powdes_certain} as given in Theorem \ref{thrm:opt:extract}; the regularized solutions converge to a true solution as \(\eps \downarrow 0\). Specifically, we set the regularization parameter as \(\eps_k \Let a/5^{k-1}\), with \(a = 0.1\). (b) Constraint satisfaction of the control trajectory obtained as the solution to OCP~\eqref{eq:energy_des_active}, where the nonconvex constraint becomes active over a certain time interval.}  
    \label{fig:combined_results_2}
\end{figure*}
The fuel-optimal trajectories and the corresponding control norms are depicted in Fig. \ref{fig:combined_results}, demonstrating constraint satisfaction and efficient landing manoeuvre. Fig. \ref{fig:reg_control} depicts the convergence and extraction of accurate optimizers for a sequence of problems where the regularization parameter was selected as \(\eps_k \Let \tfrac{a}{5^{k-1}}\) with \(a=0.1\). Note that the optimal cost obtained by solving problem \eqref{eq:powdes_certain} using \(\glbopt\) is \(152.228\), which is, as guaranteed, higher than the optimal cost value \(152.218\) obtained by solving it at uniformly distributed grid points, as done in \cite{ref:MWHBA-14b}. Error trajectory between \(s(\cdot)\) and the norm of the control is shown in Fig. \ref{fig:error_traj} (in yellow), which remains of the order of \(10^{-8}\).

\subsection{Energy optimal landing, case II (active nonconvex constraints)}\label{subsec:energy_active_control}
We consider the same energy-optimal landing problem as in \eqref{eq:energy_des_certain}, but with different boundary conditions and control bounds.  
We consider the following OCP
\begin{equation}
	\label{eq:energy_des_active}
\begin{aligned}
&\hspace{-2mm} \inf_{\cont(\cdot)}	&& \hspace{-2mm}\int_{\tinit}^{\horizon} \norm{u(t)}^2 \odif{t} \\
&  \hspace{-2mm}\sbjto		&&  \hspace{-4mm}\begin{cases}
\dot{r}(t) = v(t),\,\, \dot{v}(t) = -g + u(t),\\
 r(0) = [400, 400, 700]^{\top} \text{ m}, \,\, v(0) = -[10, 10, 105]^{\top} \text{ m/s}, \\
 r(T) = [0, 0, 0]^{\top} \text{ m}, \,\, v(T) = [0, 0, 0]^{\top} \text{ m/s}\\
\text{and safety constraints } \eqref{eq:planar_constraint} \text{ for all }t\in \lcrc{0}{\horizon}, \\  
 \norm{u(t)} \in [6.3, 10] \text{ for all } t \in \lcrc{0}{\horizon}.\\
\end{cases}
\end{aligned}
\end{equation}
In OCP \eqref{eq:energy_des_active}, after convexification the control \( \cont(\cdot) \) and slack \( \slack(\cdot) \) trajectories are parametrized as in \eqref{e:cont_param} and \eqref{e:slack_param}, with \(N = 200\) and \(\horizon = 22s\) and thus, the number of time instants for maximizing \(\mathcal{G}(\cdot;\param, 0)\) in  \(\glbopt\) is \(\dvar = 200 \times 3 + 200 = 800\).
The corresponding numerical results are presented in Fig. \ref{fig:active_control}, where it can be observed that the control norm magnitude hits the maximum limit. The green trajectory in Fig. \ref{fig:error_traj} represents the error between \(\slack(\cdot)\) and the norm of the control trajectory, which remains of the order of \(10^{-10}\).
\begin{figure}[h]
    \centering
    \includegraphics[width=\linewidth]{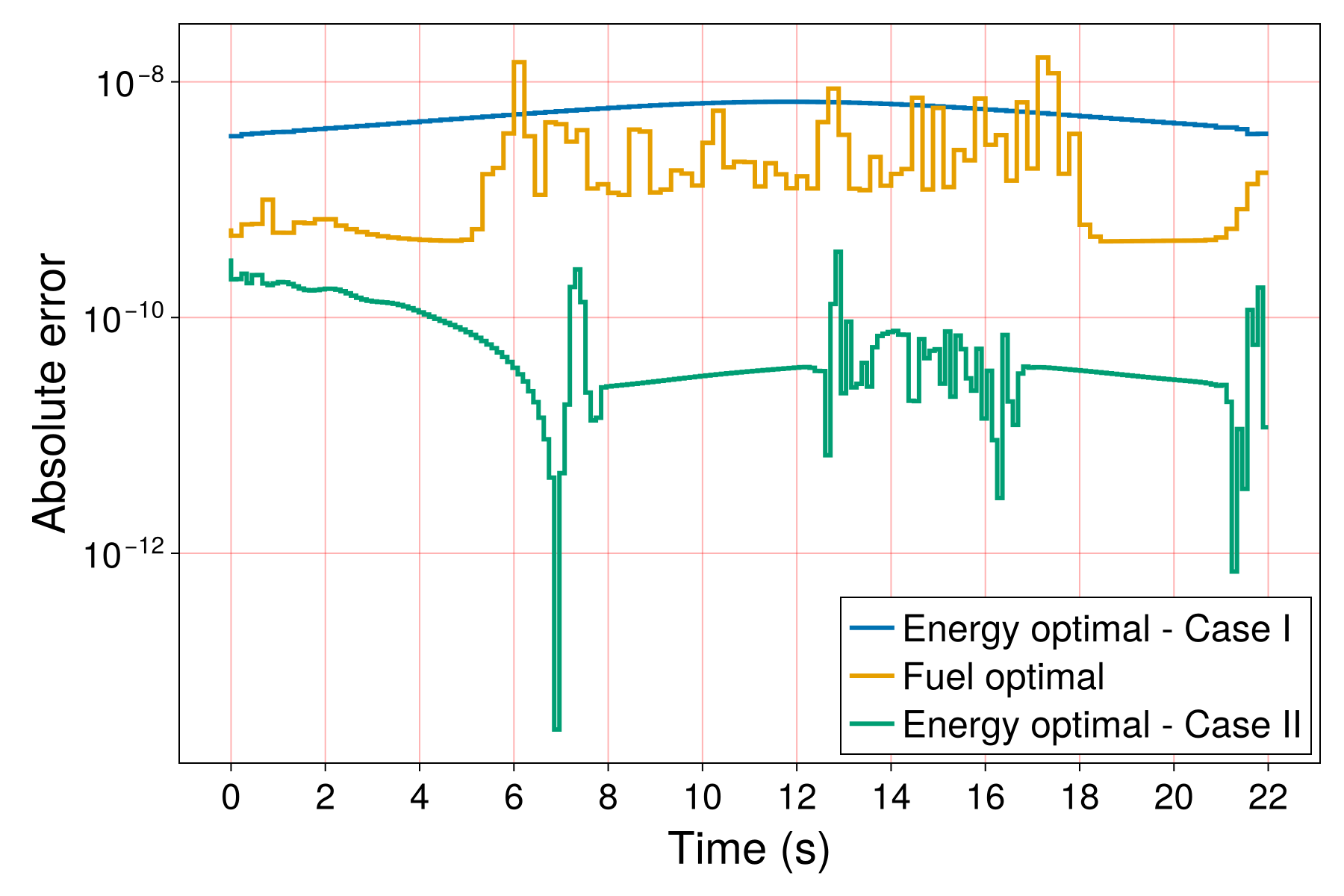}
    \caption{This plot shows the error (on log scale) between the slack trajectory and the norm of the control trajectory over time, indicating that the solution closely satisfies the necessary condition \( \kappa(\cont(t)) = \slack(t) \) for all \(t \in \lcrc{0}{22}\) as stated in \cite[Lemma 10]{ref:MWHBA-14b} up to (standard) numerical errors.}
    \label{fig:error_traj}
\end{figure}

\subsection{Fuel-optimal landing with disturbance}
\label{subsec:infeas_fuel_opt}
We adopt a similar formulation as presented in \cite[Section 7, Example 1]{ref:MWHBA-14b}. The system dynamics are given by:
\begin{equation}
\label{eq:powdes_sys_dist}
    \begin{pmatrix}
        \dot{r}(t) \\
        \dot{v}(t)
    \end{pmatrix}
    =
    \begin{pmatrix}
        v(t) \\
        -g + u(t) + w(t)
    \end{pmatrix},
\end{equation}
where \(w(t) \in \mathbb{R}^3\) denotes a known and modelled disturbance, which is given by
\begin{equation}
\label{eq:disturbance}
    w(t) = [\sin(t),\ 0,\ \cos(t)]^\top \quad \text{for all } t \in [0, \horizon].
\end{equation}
We consider the following OCP
\begin{equation}
	\label{eq:disturbed_ocp}
\begin{aligned}
&\hspace{-2mm} \inf_{\cont(\cdot)}	&& \hspace{-2mm}\int_{\tinit}^{\horizon} \norm{u(t)} \odif{t} \\
&  \hspace{-2mm}\sbjto		&&  \hspace{-4mm}\begin{cases}
\text{dynamics }\eqref{eq:disturbance},\,
\text{boundary condition } \eqref{eq:boun_cond},\\
\text{safety constraints } \eqref{eq:planar_constraint},\,\text{and }
 \norm{u(t)} \in [2, 10] \\ \text{ for all } t \in \lcrc{0}{\horizon}.\\
\end{cases}
\end{aligned}
\end{equation}
In OCP~\eqref{eq:disturbed_ocp}, after convexification the control \( \cont(\cdot) \) and slack \( \slack(\cdot) \) trajectories are discretized according to the parametrizations in \eqref{e:cont_param} and \eqref{e:slack_param}, with \(N = 100\) intervals over a horizon of \(\horizon = 21\,\text{s}\).

If the constraints are enforced only at the 100 uniformly distributed grid points --- the optimization problem appears feasible, as verified by solving the OCP \eqref{eq:disturbed_ocp} using \texttt{IPOPT} solver \cite{ref:AWLTB-06} and the \texttt{JuMP} package \cite{ref:IDJHML-2017}.\footnote{The number of time instants required for maximizing \(\gfunc(\cdot;\param, \eps)\) is \(\dvar = 100 \times 3 + 100 = 400\), as determined by \ref{dvar}. However, for the purpose of comparison with \cite[Section 7, Example 1]{ref:MWHBA-14a}, we set \(\dvar = N = 100\).} During the course of maximization of \(\gfunc(\cdot;\param, 0)\), we noted that, for
\begin{quote}
  \(\tseq' \Let\)  \texttt{\{0.21, 0.42, 0.63, 0.84, 1.05, 1.26, 1.47, 1.68, 1.89, 2.1, 2.31, 2.52, 2.73, 2.94, 3.15, 3.36, 3.57, 3.78, 3.99, 4.2, 4.41, 4.62, 4.83, 5.04, 5.25, 5.46, 5.67, 5.88, 6.09, 6.3, 6.51, 6.72, 6.93, 7.14, 7.35, 7.56, 7.77, 7.98, 8.19, 8.4, 8.61, 8.82, 9.03, 9.24, 9.45, 9.66, 9.87, 10.08, 10.29, 10.5, 10.71, 10.92, 11.13, 11.34, 11.55, 11.76, 11.97, 12.18, 12.39, 12.6, 12.81, 13.02, 13.23, 13.44, 13.65, 13.86, 14.07, 14.28, 14.49, 14.7, 14.91, 15.12, 15.33, 15.54, 15.75, 15.96, 16.17, 16.38, 16.59, 16.8, 17.01, 17.22, 17.3422225093121, 17.43, 17.64, 17.85, 18.06, 18.27, 18.48, 18.69, 18.9, 19.11, 19.32, 19.53, 19.74, 19.95, 20.16, 20.37, 20.58, 20.79\}},  
\end{quote} 
the set
\begin{align}
  \bigcap_{t \in \tseq'} \left\{\hspace{-1mm}(\Param, \beta) \middle\vert \;
\begin{array}{@{}l@{}}
       \text{boundary condition} \eqref{eq:boun_cond}, \\ 
    \inprod{\beta}{e_k} \in [2, 10],\, \kappa\left(\Param e_k \right) \leq \inprod{\beta}{e_k} \\ \text{ for all }k \in [1;100] \text{ and } r_1(t) = r_2(t)
        \end{array}
        \right\} = \emptyset. \nn
\end{align}
This shows that \(\glbopt\) can detect empty feasible sets, a feature encoded into Theorem \ref{thm:sampling_thm} and Theorem \ref{thrm:opt:extract}. Even without the verification of Slater's condition for \eqref{e:fOCP}, performing \eqref{eq:sup_sip} often reveals infeasibility over \(\dvar\) many constraints, and that immediately implies that the (uncountably constrained) problem \eqref{e:fOCP} must be infeasible. 

\subsection{Comparison}\label{subsec:comparision}
This section records several comparative results for existing methods applied to OCP \eqref{eq:energy_des_active}; this includes existing robust optimization techniques and direct trajectory optimization algorithms. These statistics are given in Table \ref{tab:solver_time_comparison_r3}.

\begin{table*}[h!]
\begin{tblr}{l c c c}
\hline[2pt]
Solver & CPU time (in secs) & Optimal value  & continuous-time constraint satisfaction\\
\hline[2pt]
\SetRow{azure9}
\(\dlcvx\): 200 grid points &  \(1.54\) & \(1612.052833\)  & \xmark \\
\SetRow{azure9}
\hspace{10.0mm} 1000 grid points &  \(36.63\) & \(1603.223254\)  & \xmark \\
\SetRow{azure9}
\hspace{10.0mm} 3000 grid points &  Solver crashed & Solver crashed  & \xmark \\
\hline
\SetRow{green9}
\(\glbopt\): Diff. evol. &  \(128.96\) & \(1601.065788\)  & \cmark \\
\SetRow{green9}
\hspace{11mm} Sim. anneal. &  \(86.62\) & \(1601.091874\)  & \cmark \\
\SetRow{green9}
\hspace{11mm} Sim. anneal. (w.s.) &  \(48.89\) & \(1601.065646\)  & \cmark \\
\SetRow{green9}
\hspace{11mm} \(\gradol\) &  \(7.84\) & \(1601.065663\)  & \cmark \\

 \hline
Scenario: 800 samples &  \(2.27\) & \(1601.065663\)  & \xmark \\
 \hline
\(\quito\): 200 grid points &  \(1.10\) & \(1612.052836\)  & \xmark \\
\hspace{10.8mm} 1000 grid points &  \(10.83\) & \(1603.223255\)  & \xmark \\
\hspace{10.8mm} 3000 grid points &  \(31.02\) & \(1601.770678\)  & \xmark \\
\hspace{10.8mm} 5000 grid points &  \(54.86\) & \(1601.480808\)  & \xmark \\
\hline[2pt]
\end{tblr}
\centering
\vspace{2.5mm}
\caption{Numerical statistics of several algorithms including \(\glbopt\) (with three modules --- differential evolution, simulated annealing, and \(\gradol\)), the scenario robust optimization techniques, \(\quito\) (a direct trajectory optimization solver), and \(\dlcvx\) (a lossless convexification method for discrete-time optimal control problems). For each solver, the reported ``CPU time (in secs)" corresponds to the average over 200 independent runs, while the reported ``Optimal value'' is the maximum objective value attained across these 200 runs. For all solvers, every minimization problem is solved using \texttt{IPOPT} in order to ensure uniformity across implementations. In the Sim. anneal. (w.s.) column of the table w.s. represents the warm start.}
\label{tab:solver_time_comparison_r3}
\end{table*}

We explain the statistics reported in Table \ref{tab:solver_time_comparison_r3}:
\begin{itemize}[leftmargin=*,label= \(\circ\)]

\item \(\dlcvx\) \cite{ref:Dlcvx-26}: \(\dlcvx\) extends the lossless convexification framework \cite{ref:MWHBA-14b} to discrete-time optimal control problems. Accordingly, a direct algorithmic comparison is not strictly appropriate, since the two approaches are developed under different problem formulations and theoretical aims. Nonetheless, we solved our numerical example using \(\dlcvx\) and report the corresponding statistics in Table~\ref{tab:solver_time_comparison_r3}. We employ an Euler discretization of the system dynamics to obtain the corresponding discrete-time model. We emphasize, as pointed out in \S4.4 of the revised manuscript, \(\dlcvx\) may not detect infeasibility. As expected, the resulting ``Optimal value” is substantially higher than that obtained with \(\glbopt\) and the scenario-based method.

\item \(\glbopt\) is our primary \emph{plug-and-play} algorithmic framework for solving the CSIP (19), which allows the flexible inclusion of any reasonable global optimization solver.
The current version consists of three modules
\begin{enumerate}[leftmargin=*]

\item Differential evolution: We utilized the differential evolution global optimization routine for the outer global maximization (\textsf{S3}) in our initial submission. 
\item Simulated annealing: In the revised version, we also tested the numerical example with simulated annealing as the global solver within the \(\glbopt\) module. From Table \ref{tab:solver_time_comparison_r3}, it can be seen that simulated annealing requires a computation time of 86.62 secs, compared to 128.96 secs for differential evolution. Moreover, when a warm start strategy is implemented, the computation time of Sim. anneal. variant of \(\glbopt\) reduced to 48.89 secs. In this implementation, the solution of the convex minimization problem obtained at the previous iterate of the global maximization step is used as the initial point for solving the subsequent convex minimization problem.
\item \(\gradol\): We employed the numerical gradient-based method \cite{ref:GradOL} in our \(\glbopt\) architecture as a maximization routine (\textsf{S3}).\footnote{\(\gradol\) is a gradient-based method for finding the smallest enclosing ball of a set (the Chebyshev center) and, more generally, for solving convex semi-infinite optimization problems.} Since it's a gradient-based method, its performance depends on the initialization; therefore, we executed it from 200 randomly sampled initializations.
\end{enumerate}

\item The scenario optimization technique \cite{MC-SG:18}: Scenario-based robust optimization techniques randomly sample a finite sequence of time instances and solve the robust optimization problems with constraints only on those time instances. Such a technique is rooted in randomization and provides probabilistic guarantees on the continuous-time constraint satisfaction, and does not provide any theoretical guarantee of satisfying \emph{all} the uncountably many constraints.

\item \(\quito\) \cite{ref:QuITOv2,ref:QuITO:SoftX}: We compared our results with direct optimization based methods. \(\quito\) is a direct transcription technique that employs the conventional discretize-the-optimize maxim to transcribe a given optimal control problem to a finite-dimensional optimization problem (nonlinear and nonconvex, in general) by discretizing the time horizon and enforcing constraints on the grid points. Due to the very nature of enforcement of constraints (in particular, safety constraints) at discrete points, there is no guarantee of constraint satisfaction for all time \(t \in \lcrc{0}{\horizon}\), which may cause inter-sample constraint violations. 
We numerically tested our example by applying \(\quito\) and observed that when the number of uniform grid points is less, the computation time is better, but the optimal value is inaccurate and often higher. With an increased number of grid points at the expense of computation time, the optimal value improves as observed in Table \ref{tab:solver_time_comparison_r3}. 
\end{itemize}

\textbf{In summary:} \(\glbopt\) delivers the strongest outcome --- exact recovery for the CSIP formulation together with \textbf{continuous-time constraint satisfaction} as per our main theoretical results (Theorem 8 and 11) and it can also \textbf{detect infeasibility} by identifying violating time instants, an example of such case is shown in \S4.4.

While with Diff. evol. and Sim. anneal., \(\glbopt\) is typically slower compared to other techniques, with the gradient-based maximization scheme \(\gradol\), the computational speed of \(\glbopt\) is of the same order as that of \(\dlcvx\) and \(\quito\) for \(N = 200\) grid points. At this discretization level, Table \ref{tab:solver_time_comparison_r3} shows a noticeable gap in the reported objective values. This discrepancy is attributable to the discretization error inherent in \(\dlcvx\) and \(\quito\), which enforces constraints only at the grid points.

Moreover, as the mesh is refined, the objective values obtained by \(\dlcvx\) and \(\quito\) approach those produced by \(\glbopt\). However, this improvement comes at the expense of increased computational effort due to the growth in problem dimension. In this refined regime, the computational time of \(\dlcvx\) and \(\quito\) becomes comparable to that observed for the simulated annealing variant of \(\glbopt\).
These observations indicate that, when comparable levels of accuracy are required, the proposed architecture remains computationally competitive with direct discretization-based optimization methods.